\documentclass[a4paper,12pt]{article}
\usepackage{amsmath,amssymb,amsthm}
\usepackage{setspace}
\usepackage{graphicx}
\usepackage{wrapfig}
 \newtheorem{theorem}{Theorem}

 \newtheorem{lemma}[theorem]{Lemma}
 \newtheorem{proposition}[theorem]{Proposition}
 
 {\theoremstyle{definition}
 
 \newtheorem{remark}{Remark}

 }
 
\usepackage{tempstq}

\def\Z{\ensuremath{\mathbb Z}} 

\newcommand{\GGG}{\Gamma}

\newcommand{\PPP}{\Psi}
\newcommand{\aaa}{\alpha}

\newcommand{\app}{\approx}
\newcommand{\bbb}{\beta}

\newcommand{\capp}{{\cal P}}

\newcommand{\dd}{\partial}

\newcommand{\ds}{\displaystyle}

\newcommand{\ess}{\emptyset}
\newcommand{\fff}{{\varphi}}

\newcommand{\mmm}{\mu}

\newcommand{\MM}{\mathfrak{M}}

\newcommand{\oo}{\infty}

\newcommand{\uI}{{\underline{I}}}

\newcommand{\peq}{\preceq}
\newcommand{\seq}{\succeq}

\newcommand{\sm}{\setminus}

\newcommand{\sss}{\sigma}

\newcommand{\TTT}{\Theta}

\newcommand{\tuM}{\widetilde{{\underline{M}}}}

\newcommand{\uM}{\underline{M}}
\newcommand{\uA}{\underline{A}}

\renewcommand{\ggg}{\gamma}

\renewcommand{\lll}{\lambda}
\title{Equi-topological entropy curves for skew tent maps in the square}
\author{
Zolt\'an Buczolich\thanks{Both 
authors  were supported by the Hungarian
National Foundation for Scientific Research Grant K104178.
},
Department of Analysis, E\"otv\"os Lor\'and\\
University, P\'azm\'any P\'eter S\'et\'any 1/c, 1117 Budapest, Hungary\\
email: buczo@cs.elte.hu\\
{\tt www.cs.elte.hu/\hbox{$\sim$}buczo}\\
 \\ and\\
 \\
Gabriella Keszthelyi$^{*}$\!\!,
 Department of Analysis, E\"otv\"os Lor\'and\\
University, P\'azm\'any P\'eter S\'et\'any 1/c, 1117 Budapest, Hungary\\
email: keszthelyig@gmail.com\\
}

\begin{document}

\maketitle
\begin{abstract}
We consider skew tent maps $T_{ { \alpha }, { \beta }}(x)$ such that  $(\aaa,\bbb)\in[0,1]^{2}$ is the turning point of $T\saab$, that is, $T_{ { \alpha }, { \beta }}=\frac{ { \beta }}{ { \alpha }}x$ for $0\leq x \leq { \alpha }$ and $T_{ { \alpha }, { \beta }}(x)=\frac{ { \beta }}{1- { \alpha }}(1-x)$
for $ { \alpha }<x\leq 1$. 
 We denote by $\uM=K(\aaa,\bbb)$ the kneading sequence of
 $T\saab$ and by $h(\aaa,\bbb)$ its topological entropy.
 For a given kneading squence $\uM$ we consider equi-kneading, (or equi-topological entropy, or isentrope) curves $(\aaa,\fff_{\uM}(\aaa))$  such that $K(\aaa,\fff_{\uM}(\aaa))=\uM$. To study the behavior of these curves
 an auxiliary function $\TTT_{\uM}(\aaa,\bbb)$ is introduced.
 For this function $\TTT_{\uM}(\aaa,\fff_{\uM}(\aaa))=0$,
 but it may happen that for some kneading sequences $\TTT_{\uM}(\aaa,\bbb)=0$ for some $\bbb<\fff_{\uM}(\aaa)$ with
 $(\aaa,\bbb)$ still in the dynamically interesting quarter of the unit square. Using $\TTT_{\uM}$ we show that the curves $(\aaa,\fff_{\uM}(\aaa))$ hit the diagonal $\{ (\bbb,\bbb): 0.5<\bbb<1 \}$ almost perpendicularly if $(\bbb,\bbb)$ is close to $(1,1)$.
 Answering a question asked by M. Misiurewicz at a conference we show that  these curves are not necessarily exactly orthogonal to the diagonal, 
 for example for
 $\uM=RLLRC$ the curve $(\aaa,\fff_{\uM}(\aaa))$
 is not orthogonal to the diagonal. On the other hand,
 for $\uM=RLC$ it is.\\
With different parametrization properties of equi-kneading maps for skew tent
maps were considered by J.C. Marcuard, M. Misiurewicz and  E. Visinescu. 
\end{abstract}

\section{Introduction}

Consider a point $(\alpha,\beta)$ in the unit square $[0,1]^2$. Denote by $T_{\alpha,\beta}(x)$ the skew tent map.
\begin{equation}\label{021602a}
T_{\alpha,\beta}(x) =\left\{
\begin{array}{clcr}
			\frac{\beta}{\alpha}x & {\rm if}    & 0 \leq x \leq \alpha,      \\
			\frac{\beta}{1-\alpha}(1-x)    &  {\rm if}  & \alpha< x \leq 1.
	\end{array}
\right.
\end{equation}
\begin{center}
\begin{figure}
\includegraphics{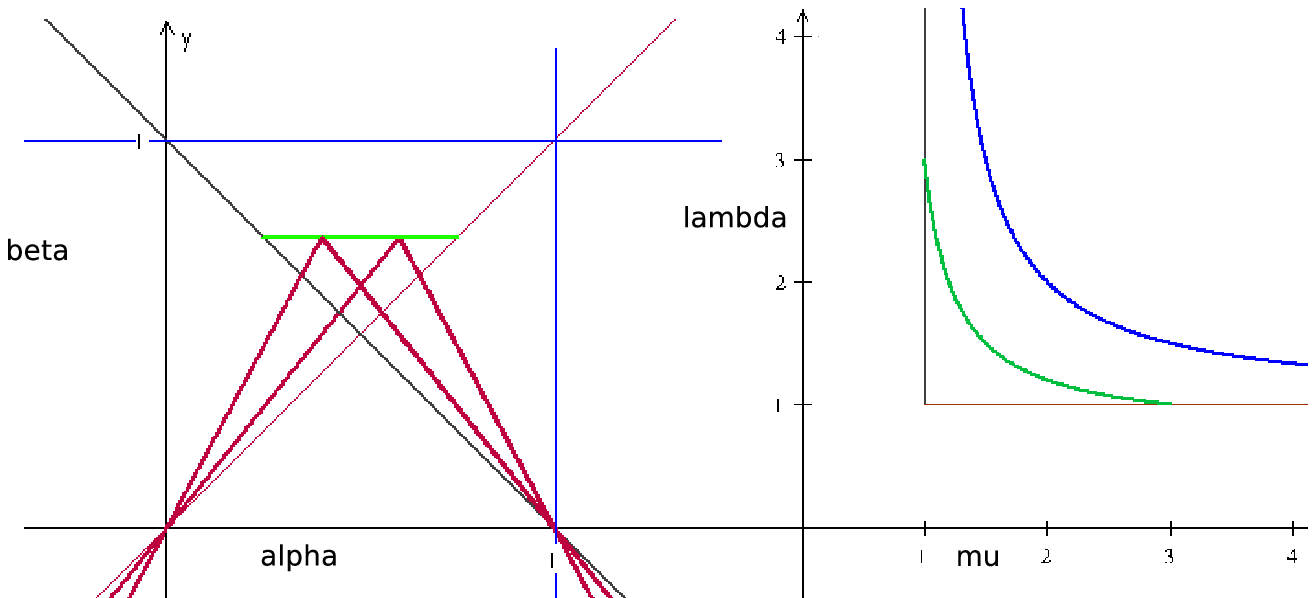}
\caption{{\bf To the left:} Skew tent maps in the $\aaa$-$\bbb$ phase space;
{\bf To the right:} the $\mmm$-$\lll$ parameter space} 
\label{fig0*}
\end{figure}
\end{center}
The topological entropy of $T_{\alpha,\beta}$ is denoted by $h(\alpha,\beta)$. To avoid trivial dynamics we suppose that $0.5< \beta\leq 1$ and
$\alpha \in (1-\beta,\beta)$. We denote by $U$ the region of $[0,1]^2$ consisting of these $[\alpha,\beta]$. Suppose $\beta$ is fixed. The starting
point of our paper was the question about the behavior of the function $h(\alpha)=h(\alpha,\beta).$ The answer to the question about the behavior of
the function $g(\beta)=h(\alpha,\beta)$ with a fixed $\alpha$ is known. In the case of these $\beta's$ where the dynamics of $T_{\alpha,\beta}(x)$ is
nontrivial the function $g(\beta)$ is monotone increasing. Skew tent maps and topological entropy were considered by M. Misiurewicz and E. Visinescu in \cite{[MiVi]}. In this paper a different
parametrization was used. The functions
$$F_{\lambda,\mu}(x) =\left\{
\begin{array}{clcr}
			{1+\lambda}x & {\rm if}   & x \leq 0      \\
			1-\mu x    &  {\rm if}  & x \geq 0
	
		\end{array}
\right.$$
were considered on $\mathbb{R}$. It is rather easy to see that if $\lambda=\frac{\beta}{\alpha}, \; \mu=\frac{\beta}{1-\alpha}$ and $h(x)=(\beta-\alpha)x + \alpha, \; h^{-1}(x)=\frac{x-\alpha}{\beta-\alpha}$ then
\begin{equation}\label{intro2a}
F_{\lambda,\mu}(x)=(h^{-1} \circ T_{\alpha,\beta} \circ h)(x),
\end{equation}
provided that we extend the definition of $T_{\alpha,\beta}$ onto $\mathbb{R}$ in the obvious way.

\begin{wrapfigure}{l}{0pt}
\includegraphics[width=7.0cm]{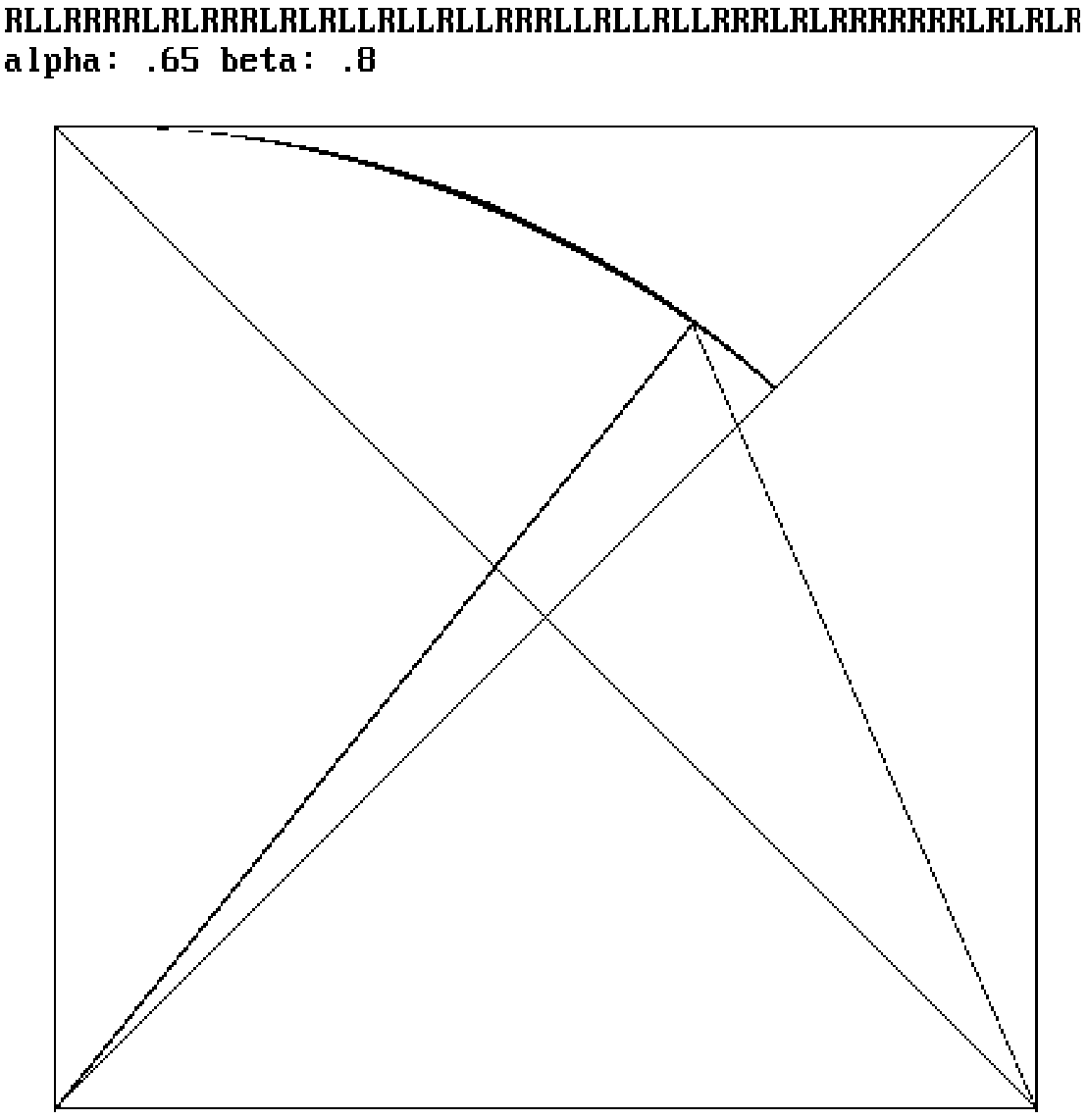}
\caption{Equi-kneading curve for $\aaa=.65$ and $\bbb=.8$} 
\label{fig1*}
\end{wrapfigure}
Results of \cite{[MiVi]} imply that if
$\lambda' \geq \lambda, \; \mu' \geq \mu$ and at least one of these inequalities is sharp, then $h(F_{\lambda',\mu'})>h(F_{\lambda,\mu})$ where
$h(F_{\lambda,\mu})$ denotes the topological entropy of $F_{\lambda,\mu}$. If $\beta$ is fixed and $\alpha$ increases then
$\lambda=\frac{\beta}{\alpha}$ decreases, while $\mu=\frac{\beta}{1-\alpha}$ increases. This implies that the monotonicity  result in \cite{[MiVi]} is not
giving an answer to our question. 
In fact, points $\lambda=\lambda(\alpha,\beta)$ and $\mu=\mu(\alpha,\beta)$ with fixed $\beta$ satisfy
$\frac{1}{\lambda}+\frac{1}{\mu}=\frac{1}{\beta}$, or $\lambda=\frac{1}{\frac{1}{\beta}-\frac{1}{\mu}}$ and hence $\lambda$ is a monotone decreasing
function of $\mu$ along curves in the $(\mu,\lambda)$-plane corresponding to horizontal line segments in the $(\alpha,\beta)$-plane. If $\alpha$ is fixed
and $\beta$ increases then both $\lambda=\frac{\beta}{\alpha}$ and $\mu=\frac{\beta}{1-\alpha}$ increase. 
Therefore, the monotonicity result in \cite{[MiVi]} implies
that all the functions $g(\beta)=h(\alpha,\beta)$ are monotone increasing in our parameter range. It is also worth mentioning that our parameter range $0.5<\beta<1, \; \alpha \in (1-\beta,\beta)$
corresponds to the region bounded by the curves  $\mu=1, \; \lambda=1$ and $\frac{1}{\mu}+\frac{1}{\lambda}=1$ in the $(\mu,\lambda)$-plane. See
Figure \ref{fig0*} (we use the $(\mu,\lambda)$-plane to be consistent with the figure on p.129 of \cite{[MiVi]}).  
On the left half of this figure two tent maps are shown with the same
$\bbb$ parameter. On the right side this constant $\bbb$ curve is also shown.\\
On Figure \ref{fig1*} the equi-kneading region of an individual tent map is illustrated.
On this figure on the top line one can see the kneading sequence calculated by the computer, then $T_{.65,.8}$ is plotted and some pixels with similar first few initial kneadig sequence entries.
On the left half of Figure \ref{fig22*} one can see the region $U$ colored in a way that one similarly colored/shaded connected region means that for $(\aaa,\bbb)$'s in one such region the first eight entries of the kneading squence $K(\aaa,\bbb)$ are the same. On the right half of this figure one can see the image of these regions if the $\lll-\mmm$ parametrization is used. 

To study equi-topological entropy, or equi-kneading curves in the region $U$ we introduce
the auxiliary functions $\TTT_{\uM}$.
Suppose that we have a  given kneading-sequence $\uM$ and 
\begin{equation}\label{041801aa}
\underline{M}^{-}=R\underbrace{L\dots L}_{m_1}R\underbrace{L\dots L}_{m_2}R\underbrace{L\dots L}_{m_3}R \dots .
\end{equation}
We put $\omm_{k}=m_1+m_2+\dots+m_k$ and
\begin{equation}\label{02261a}
  \Theta_{\underline{M}}(\alpha,\beta)=1-\beta+\sum_{k=1}^{\infty}(-1)^{k}\left(\frac{1-\alpha}{\beta}\right)^k\left(\frac{\alpha}{\beta}\right)^{{\overline m}_k}.
\end{equation}
In Theorem \ref{thx} we show
that for $(\aaa,\bbb)\in U$ it follows from $K(\aaa,\bbb)=\uM$ that
$\TTT_{\uM}(\aaa,\bbb)=0$. This means that the equi-topological entropy curve
$\{ (\aaa,\bbb)\in U : K(\aaa,\bbb)=\uM \}$ is a subset of
$\{ (\aaa,\bbb)\in U: \TTT_{\uM}(\aaa,\bbb)=0 \}$, the zero level set of $\TTT_{\uM}$. Based on some computer simulations, it might appear that this level set coincides in $U$ exactly with the equi-topological entropy curve, but this is not true for all parameter values. For example on Figure  
\ref{theta78*} one can see the level set structure  of $\TTT_{K(0.5,0.7)}$ and $\TTT_{K(0.5,0.8)}$ in the upper half of $[0,1]^{2}$.
Similarly colored/shaded connected regions mean not too different $\TTT$ values. The zero level consists of some  quite clearly defined common boundary curves of some of these regions.  
In Theorem \ref{thex} we show that there exists $\uM$ and $(\aaa,\bbb)\in U$
 such that $K(\aaa,\bbb)=\tuM\not=\uM$, but Theorem \ref{felsozona}
 implies that this can happen only when  $K(\aaa,\bbb)=\tuM\prec \uM$.
 On Figure \ref{theta34*} one can see the level set structure of $\TTT$ corresponding to some of these exceptional parameter values.
 Analytic properties of $\TTT_{\uM}$ can help to study the behavior of
 equi-kneading curves $\PPP_{\uM}$ and the functions $h(\aaa)=h(\aaa,\bbb).$
Looking at the equi-kneading curves on the left half of Figure \ref{fig22*} it seems that they are almost perpendicular
 to the diagonal $\aaa=\bbb$ when we are close to $(\aaa,\bbb)=(1,1)$,
 in fact the angle formed by the equi-kneading curve and the diagonal
 tends to right angle as $(\aaa,\aaa)\to (1,1)$.
 To verify this almost perpendicularity property we need further estimates of $\TTT_{\uM}$
 in Section \ref{secfurTTTM}, while in Section \ref{secortho} we prove the almost orthogonality
 property. 
 Since the first differential of $\TTT_{\uM}$ vanishes at these intersections with the diagonal one cannot use implicit differentiation. Instead of this we show that 
 near $(1,1)$ at these intersection saddle points the second differential approximates constant times $x^2-y^2$.
 M. Misiurewicz asked the first author at a conference whether we have
 almost orthogonality, or exact orthogonality. 
 Apparently, as Figure \ref{fig505*} illustrates close to $(1/2,1/2)$
these equi-kneading curves does not seem to be perpendicular to the diagonal. 
 We also see in this Section \ref{secortho} that for $\uM=RLC$
 we have exact orthogonality, while for $\uM=RLLRC$ the slope of the tangent at the intersection
 point with the diagonal equals $\ds -{\sqrt{5}+3\over 2 \sqrt{5}+2}\not=-1$,
 showing that in this case we do not have exact orthogonality.
 See Figure \ref{thetarllrc*}.
 Again $\TTT_{\uM}$ has a saddle point,
 with vanishing first differential at these intersection points
 and hence one cannot use implicit differentiation to compute these tangents.
\begin{center}
\begin{figure}
\includegraphics[height=8.2cm]{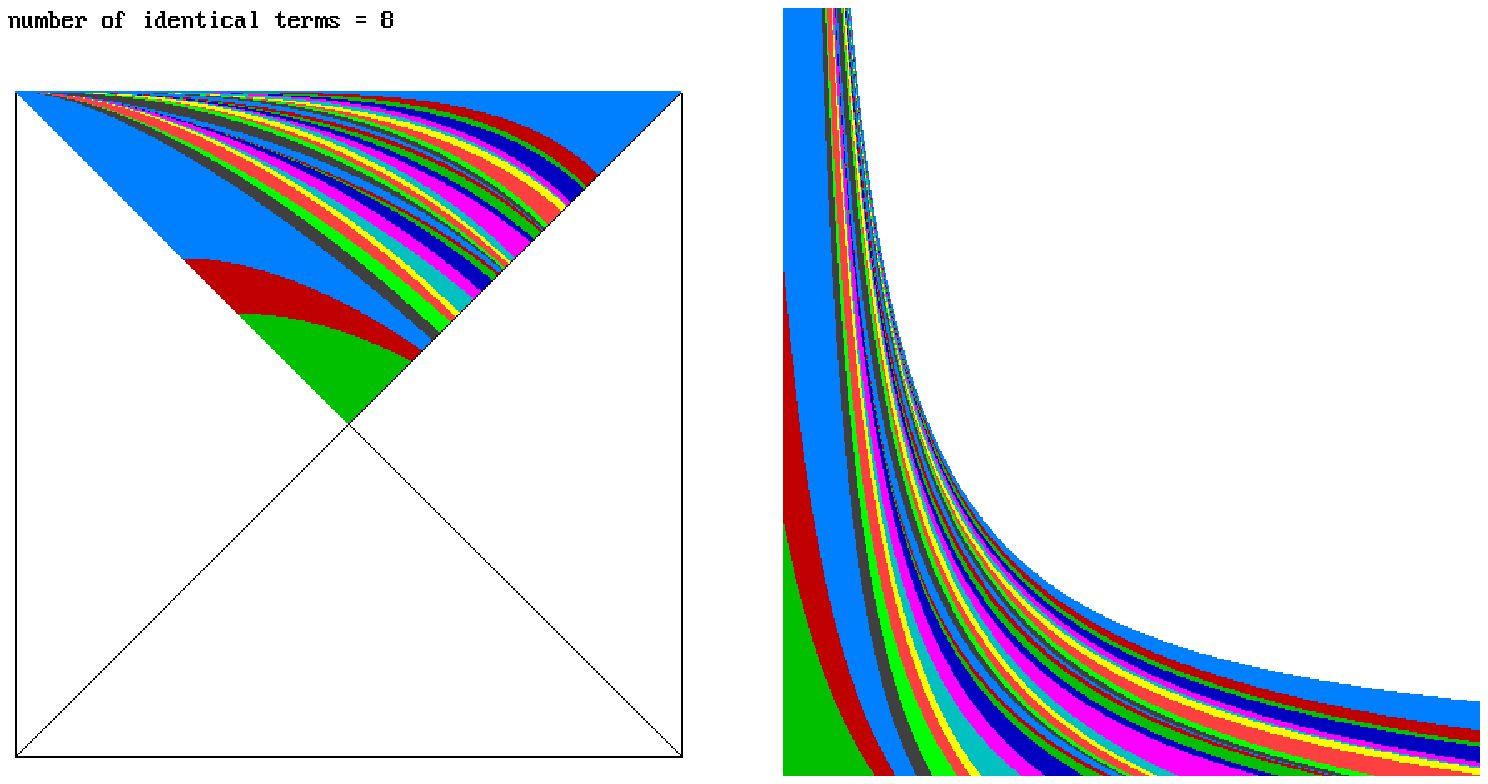}
\caption{equi-kneading regions in the square and in the $\lll-\mmm$ parameter space} 
\label{fig22*}
\end{figure}
\end{center}
This paper is the first publication about our research project. The proof of the conjecture/theorem stating that the  equi-kneading curves are strictly monotone decreasing, which implies that the topological entropy is monotone decreasing along horizontal lines for certain parameter regions is far from trivial
and quite technical. At this time we still try to simplify these lengthy
arguments which are for certain parameter values computer assisted.
For certain critical parameter values (where the infinitely small errors exclude numerical calculations) and for some easy to handle parameter domains we have pure theoretical arguments.  Adding them to the current paper would make this paper
way too lengthy and technical.  

In our paper \cite{[SNR]} we plan to deal with the easy subregions of $U$, these are the at least twice renormalizable systems and the "right half" of the once renormalizable region. In this paper we also plan to discuss results concerning the non-renormalizable systems.

In \cite{[SOR]} we plan to discuss the technically most complicated parameter region, the "left half" of the once renormalizable region.

Some chaotic regions in the parameter space of skew tent maps were considered by \cite{[Bas]}.
Monotonicity of topological entropy for biparametric families, especially
for skew tent maps were considered by several authors for example
see \cite{[AlMa]}, \cite{[Ic]},  \cite{[IcIt]}, \cite{[MaVi]} and \cite{[MiVi]}.
Some further results about skew tent maps (with different parametrizations):
In \cite{[BiBo]}
Markov property and invariant densities of skew tent maps are studied and 
kneading sequences are also  used.
The study of invariant densities is continued in
\cite{[MaKa]}.
In \cite{[LT]}   the authors classify the dynamics of skew tent maps in terms of two bifurcation parameters. 
In \cite{[BrvS]} equi-topological entropy regions are called isentropes and connectedness of isentropes is verified for real multimodal polynomial interval maps
with only real critical points.


\section{Preliminaries}\label{secprelim}

For $x \in \mathbb{R}$ we  put 
\begin{equation}\label{*labx}
\text{$\ds L_{\alpha,\beta}(x)=\frac{\alpha}{\beta} x$ and $\ds R_{\alpha,\beta}(x)=\frac{\beta}{1-\alpha}(1-x)$.}
\end{equation}


Kneading theory was introduced by J. Milnor and W. Thurston in \cite{[MT]}.
For symbolic itineraries and for the kneading sequences we follow the notation of \cite{[CE]}.

Suppose $T=T_{\aaa,\bbb}$ is fixed for an $(\aaa,\bbb)\in U$ and $x\in [0,1]$.
The itinerary of $x$ is denoted by $\uI(x)=\uI_{\aaa,\bbb}(x)$ and its
extended itinerary by $\uI_{E}(x)=\uI_{E,\aaa,\bbb}(x)\in \{ L,R,C \}^{\Z_{\geq 0}}$. We denote the $n$'th element of $\uI(x)$, by $I_{n}(x)$,
while the $n$'th element of  $\uI_{E}(x)$ is denoted by $I_{E,n}(x)$.
Recall that $I_{E,n}(x)=L$ if $T_{\aaa,\bbb}^n(x)<\aaa$, 
$I_{E,n}(x)=C$ if $T_{\aaa,\bbb}^n(x)=\aaa$ and
$I_{E,n}(x)=R$ if $T_{\aaa,\bbb}^n(x)>\aaa$.
If there is no $C$ in $\uI_{E}(x)$ then $\uI(x)=\uI_{E}(x)$
if there are $C$'s in $\uI_{E}(x)$  then $\uI(x)$ is a finite string which is obtained by stopping at the first $C$ and throwing away
the rest of the infinite string $\uI_{E}(x)$. 
 We denote by $K(\alpha,\beta)=\uI_{\aaa,\bbb}(\bbb)$ the kneading sequence of $T_{\alpha,\beta}$. 

If $\beta$ is fixed we simply write $K(\alpha)$ and $T_\alpha$. If $K(\alpha)=\uM=A_0A_1 \dots$ then $K_n(\alpha)=\uM_{n}=A_n$ for $n=0,1,...$.

We denote by $ { \underline { M } }|n$ the string of the first $n$ entries of $ { \underline { M } }$, that is, $ { \underline { M } }|n=A_0 \dots A_{n-1}$.

A sequence $\uM$ of symbols $L,R,C$ is called admissible if either is an infinite sequence of $L$'s and $R$'s or if $\uM$ is a finite (or empty) sequence of $L$'s and $R$'s, followed by $C$.




Recall that for a finite $\underline{A}=A_{0}...A_{k-1}$ and an arbitrary $\underline{B}=B_{0}B_{1}...$, the $*$ product $\underline{A}* \underline{B}$ is defined as follows:
\begin{itemize}
  \item If $\underline{A}$ is even (i.e. the number of $R$'s in $\uA$ is even),  then $\underline{A} * \underline{B}=\underline{A}B_0 \underline{A} B_1 \underline{A} B_2 \dots$
  \item If $\underline{A}$ is odd, then $\underline{A}*\underline{B}=\underline{A}\breve{B_0}\underline{A}\breve{B_1}\underline{A}\breve{B_2}\dots$, where $\breve{L}=R, \: \breve{R}=L, \: \breve{C}=C$.
\end{itemize}


Following notation of \cite{[MiVi]} we denote by $\mathfrak{M}$ the class of kneading sequences $K(0.5,\beta), \,\beta \in (0.5,1]$, this corresponds to the kneading sequences of functions $F_{\mu,\mu}$ with $1<\mu \leq 2$, or equivalently to kneading sequences 
of $T_{\frac{1}{2},\bbb}$ with $\frac{1}{2}<\bbb \leq 1$.

We denote by $\prec $ the symbolic ordering of itineraries.
Recall that $L\prec C \prec R$. Otherwise one is expected to find the first entry
where two itineraries, say $\underline{A}$ and $\underline{B}$ differ. Suppose this is the $n$'th entry. If there are even many $R$'s in $A_{0}...A_{n-1}$ and $A_{n}\prec B_{n}$ then $\underline{A} \prec \underline{B}$.  If there are odd many $R$'s in $A_{0}...A_{n-1}$ and 
$A_{n}\prec B_{n}$ then $\underline{A} \succ \underline{B}$.

We also remind that $\uM$  is maximal if $\uM\seq \sss ^{n}\uM$
for all $n$, where $\sss$ is the one-sided shift.

By Lemmas II.1.2. and  II.1.3. 
of \cite{[CE]} if $T$ is any unimodal mapping of $[0,1]$
then $\uI(x)\leqc \uI(x')$ implies $x<x'$, and
$x<x'$ implies $\uI(x)\peq \uI(x')$.

 A sequence $\underline{M}$ is in $\mathfrak{M}$ if
\begin{itemize}
\item I. $\underline{M}$ is a maximal admissible sequence,
\item II. $\underline{M}\succ R^{*\infty}$,
\item III. if $\underline{M}=\underline{A}*\underline{B}$ with $\underline{A} \neq \ess, \, \underline{B} \neq C$ then $\underline{A}=R^{*m}$ for some $m$. 
\end{itemize}
It is easy to see that if $(\alpha,\beta) \in U$ then $(\lambda(\alpha,\beta),\mu(\alpha,\beta))=(\frac{\beta}{\alpha},\frac{\beta}{1-\alpha})$ belongs to the region $D=\{(\lambda, \mu):\lambda \geq 1, \mu>1, \frac{1}{\lambda}+\frac{1}{\mu} \geq 1\}$ considered in the paper \cite{[MiVi]}.


By $\MM_{\oo}$ we denote those kneading sequences in $\MM$
which do not contain $\ds C$. These are the infinite sequences.
On the other hand, $\MM_{<\oo}$ will denote the finite kneading sequences.
These are the ones ending with $C$ corresponding to parameter values
when the turning point is periodic.

Suppose $K(\aaa,\bbb)=\uM\in \MM$. We put $\uM^{-}=\lim_{x\to \bbb-}
\uI_{E,\aaa,\bbb}(x)$.
If $\uM\in\MM_{\oo}$ then $\uM^{-}=\uM$.
If $\uM=A_{0}...A_{n-1}C\in\MM_{<\oo}$ with $A_{i}\in \{ L,R \}$
then for even $A_{0}...A_{n-1}$ we have $\uM^{-}=(A_{0}...A_{n-1}L)^{\oo}$,
while for odd  $A_{0}...A_{n-1}$ we have $\uM^{-}=(A_{0}...A_{n-1}R)^{\oo}$.
It is known and not difficult to see that $\uM^{-}$ is a maximal infinite string and
$\uM^{-}\peq \uM$. 
We quickly remind the reader why $\uM^{-}$ is maximal.
Indeed, if $\uM\in\MM_{\oo}$ then $\uM^{-}=\uM$ is clearly maximal since it is a kneading sequence.
 If $\uM\in\MM_{<\oo}$ then proceeding towards a contardiction suppose that $\sss^{k}(\uM^{-})\geqc \uM^{-}$ and $\sss^{k}(\uM^{-})$ first 
differs at its $j$'th entry from the $j$'th entry of $\uM^{-}$, (recall that $\sss$ denotes the shift). 
Then choose $\bbb_{1}<\bbb$ close enough to $\bbb$ such that up to the $(k+j)$'th entry $\uI_{E,\aaa,\bbb}(\bbb_{1})=\uI_{\aaa,\bbb}(\bbb_{1})\in \MM_{\oo}$
equals the corresponding entries of $\uM^{-}$. Let $x^{*}=T_{\aaa,\bbb}^{k}(\bbb_{1}).$
Then $x^{*}\not=\bbb$ and up to the $j$'th entry $\uI\saab (x^{*})$ equals $\sss^{k}(\uM^{-})$.
Choose $\bbb_{2}<\bbb$  such that $x^{*}<\bbb_{2}$ and up to the $j$'th entry $\uI\saab(\bbb_{2})=\uI_{E,\aaa,\bbb}(\bbb_{2})\in \MM_{\oo}$ equals $\uM^{-}$. Then $\uI\saab(\bbb_{2})\geqc \uI\saab(x^{*})$ and this implies $\uM^{-}\geqc \sss^{k}(\uM^{-})$, a contradiction.

It follows from Theorem II.3.8. of \cite{[CE]} that if  $T_{\aaa,\bbb}$ is one of our tent maps  and  $\uA$ is an admissible sequence satisfying
\begin{itemize}
 \item (i) $\uA\seq \uI_{\aaa,\bbb}(0)=L^{\oo}$. (This condition always holds.)
\item (ii) If $K(\aaa,\bbb)$ is infinite then $\sigma^{k }\uA\peq K(\aaa,\bbb)$ for all $k.$
\item  (iii) If $K(\aaa,\bbb)=\uD C$ is finite then $\sigma^{k}\uA\peq \inf ((\uD L )^{\oo}, (\uD R)^\oo)$ for all $k.$
 \end{itemize}
 Then 
there is an $x\in [0,1]$ such that $\uI_{\aaa,\bbb}(x) = \uA.$
\\Using notation introduced in this paper conditions (ii) and (iii) can be replaced by $\sigma^{k }\uA\peq K(\aaa,\bbb)^{-}$ for all $k.$

We denote by $\mathfrak{M}_{R \infty}$ the set of those $\underline{M} \in \mathfrak{M}$ which are of the form $\underline{M}=A_0A_1 \dots A_{n-1}R^{\infty}$.


We recall part of Theorem C from \cite{[MiVi]} with some change in notation .
\begin{theorem}\label{MV}
For each $M \in \mathfrak{M}$ there exists a number $\Gamma(\underline{M})$ and a continuous decreasing function $\varphi_{\underline{M}}:(1,\Gamma(\underline{M})] \rightarrow [1,\infty)$
(with one exception $\underline{M}=RL^{\infty}$ when $\Gamma(\underline{M})=\infty$ and the domain of $\varphi_{\underline{M}}$ is $[1,\infty)$) such that for $(\lambda,\mu) \in D$ we have $K(F_{\lambda,\mu})=\underline{M}$ if and only if $\lambda=\varphi_{\underline{M}}(\mu)$. The function $\Gamma(\underline{M})$ is increasing. The graphs of the functions $\varphi_{\underline{M}}$ fill up the whole set $D$. Moreover,
$$\lim_{\mu \rightarrow 1+0} \varphi_{\underline{M}}(\mu)=\infty \qquad \text{if} \qquad \underline{M} \seq RLR^{\infty},$$
$$\lim_{\mu \rightarrow 1+0} \varphi_{\underline{M}}(\mu)=\Gamma(\underline{J}) \qquad \text{if} \qquad \underline{M} \prec  RLR^{\infty},$$
and $\underline{J}$ is given by $\underline{M}=R*\underline{J}, \,\varphi_{\uM}(\Gamma(\underline{M}))=1$ if $\underline{M} \neq RL^{\infty}$, \\
$\lim_{\mu \rightarrow \infty} \varphi_{\underline{M}}(\mu)=+\infty$ if $\underline{M}=RL^{\infty}.$
\end{theorem}

In case we want to translate Theorem \ref{MV} for our parameter range we can state the following:
\begin{theorem}\label{MVab}
For each $\uM \in \mathfrak{M}$ there exist two numbers $\alpha_1(\underline{M}) < \alpha_2(\underline{M})$ and a continuous function $\Psi_{\underline{M}}:(\alpha_1(\underline{M}),\alpha_2(\underline{M})) \rightarrow U$ such that for $(\alpha,\beta) \in U$ we have $K(\alpha,\beta)=\underline{M}$ if and only if $\bbb=\Psi_{\underline{M}}(\alpha)$. 
The graphs of the functions $\Psi_{\underline{M}}$ fill up the whole set U. Moreover, $\lim_{\alpha \rightarrow \alpha_1(\underline{M})+} \Psi_{\underline{M}} (\alpha)=1$ if $\underline{M}\seq  RLR^{\infty}$. If $\underline{M}\prec  RLR^{\infty}$ then the curve $(\alpha, \Psi_{\underline{M}}(\alpha))$ converges to a point on the line segment $\{(\alpha,1-\alpha):0<\alpha<\frac{1}{2}\}$
as $\aaa\to\aaa_{1}(\uM)+$.
 If $\underline{M}=RL^{\infty}$ then $\alpha_1(\underline{M})=0, \, \alpha_2(\underline{M})=1$ and $\Psi_{\underline{M}}(\alpha)=1$ for all $\alpha \in (0,1).$

\end{theorem}


\section{The auxiliary functions $\TTT_{\uM}$}\label{sectheta}


\begin{center}
\begin{figure}
\includegraphics{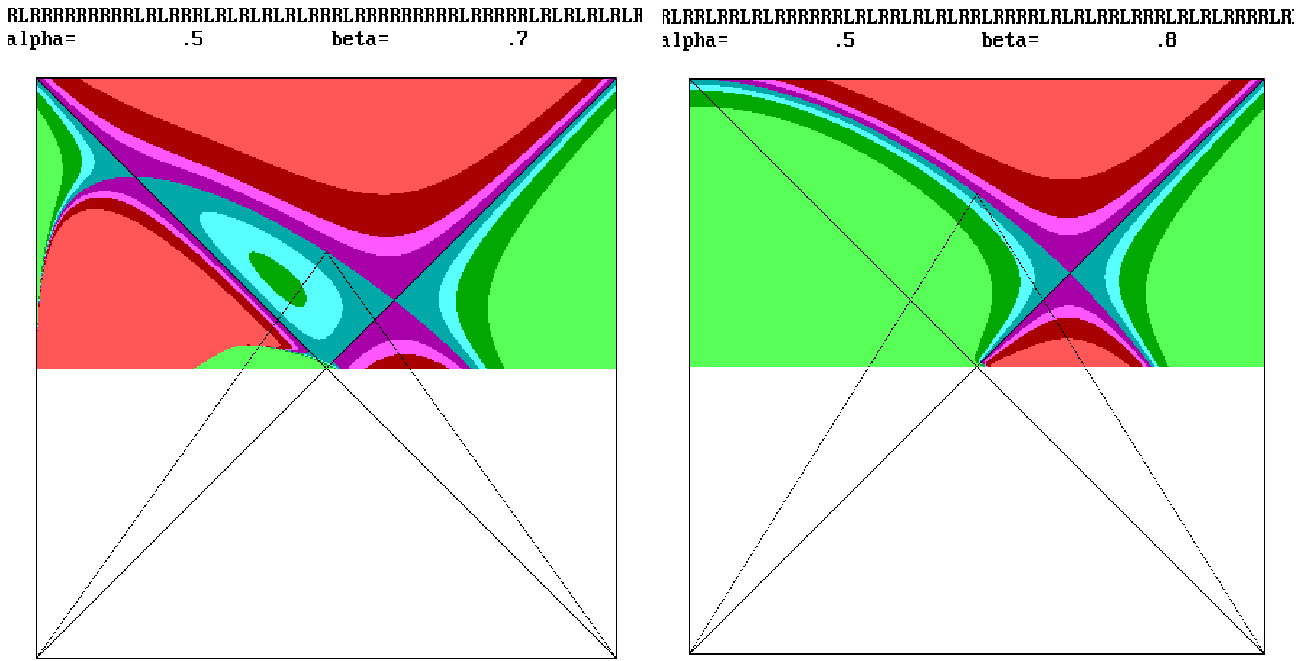}
\caption{Level set structure of $\TTT$ when $\uM=K(.5,.7)$ and when $\uM=K(.5,.8)$} 
\label{theta78*}
\end{figure}
\end{center}

Suppose 
\begin{equation}\label{041801a}
\underline{M}^{-}=R\underbrace{L\dots L}_{m_1}R\underbrace{L\dots L}_{m_2}R\underbrace{L\dots L}_{m_3}R \dots .
\end{equation}

Recall the notation $\omm_{k}=m_1+m_2+\dots+m_k$.

\begin{theorem}\label{thx}
Suppose $\underline{M} \in \mathfrak{M}\backslash\{RL^{\infty}\}$ is given. Then $\Theta_{\underline{M}}:U \rightarrow \mathbb{R}$
 defined in \eqref{02261a}  has the property that for $(\alpha,\beta) \in U$ if $K(\alpha,\beta)=\underline{M}$ then $  \Theta_{\underline{M}}(\alpha,\beta)=0.$
Moreover, $\Theta_{\underline{M}}$ 
can also be written as
\begin{equation}\label{*4b}
 \Theta_{\underline{M}}(\alpha,\beta)=1-\beta+ \sum_{k=1}^{\infty}\left(\frac{\alpha-1}{\beta}\right)^{k}\left(\frac{\alpha}{\beta}\right)^{\omm_{k}}
 \end{equation}
where $m_1={\overline m}_1>0, \, {\overline m}_k \leq {\overline m}_{k+1}\leq {\overline m}_k+{\overline m}_1, \, k=0,1,\dots .$ If $\uM=K(\alpha,\beta)\in \mathfrak{M}_{R,\infty}$ then there exists $n$  such that ${\overline m}_{k+1}={\overline m}_k$ for $k \geq n.$
\end{theorem}


\begin{proof}
By maximality of $\uM$ and $\uM^{-}$ we have $m_j \leq m_1$ for $j=1,2, \dots$\:. \\
If $T=T_{\alpha,\beta}$ then $T(\alpha)=\beta, T(\beta)=\frac{\beta}{1-\alpha}(1-\beta)$,
$T^{1+m_1}(\beta)=\frac{\beta}{1-\alpha}\left(\frac{\beta}{\alpha}\right)^{m_1}(1-\beta)$,\\
$T^{2+m_1+m_2}(\beta)=\left(\frac{\beta}{\alpha}\right)^{m_2}\frac{\beta}{1-\alpha}(1-T^{1+m_1}(\beta))=
\left(\frac{\beta}{1-\alpha}\right)^2\left(\frac{\beta}{\alpha}\right)^{\omm_2}(-1)\left(1-\beta-\left(\frac{1-\alpha}{\beta}\right)\left(\frac{\alpha}{\beta}\right)^{m_1}\right)=$ $
\left(\frac{\beta}{1-\alpha}\right)^2\left(\frac{\beta}{\alpha}\right)^{\omm_2}(-1)\cdot \Theta_{\uM,1+\omm_1}$,\\
$T^{3+\omm_3}(\beta)=\left(\frac{\beta}{1-\alpha}\right)^{3}\left(\frac{\beta}{\alpha}\right)^{\omm_3}(-1)^2\left(1-\beta-\left(\frac{1-\alpha}{\beta}\right)\left(\frac{\alpha}{\beta}\right)^{\omm_1}+\left(\frac{1-\alpha}{\beta}\right)^2\left(\frac{\alpha}{\beta}\right)^{\omm_2}\right)=\\
\left(\frac{\beta}{1-\alpha}\right)^3\left(\frac{\beta}{1-\alpha}\right)^{\omm_3} \cdot \Theta_{\uM,2+\omm_2}$.\\ \newline
In general, \\ \newline
$T^{k+\omm_k}(\beta)=\left(\frac{\beta}{1-\alpha}\right)^{k}\left(\frac{\beta}{\alpha}\right)^{\omm_{k}}(-1)^{k-1}\bigg(1-\beta-\left(\frac{1-\alpha}{\beta}\right)\left(\frac{\alpha}{\beta}\right)^{\omm_1}+\left(\frac{1-\alpha}{\beta}\right)^2\left(\frac{\alpha}{\beta}\right)^{\omm_2}-+ \dots\\+ (-1)^j\left(\frac{1-\alpha}{\beta}\right)^j\left(\frac{\alpha}{\beta}\right)^{\omm_j}+ \dots+(-1)^{k-1}\left(\frac{1-\alpha}{\beta}\right)^{k-1}\left(\frac{\alpha}{\beta}\right)^{\omm_{k-1}}\bigg)=\\
\left(\frac{\beta}{1-\alpha}\right)^{k}\left(\frac{\beta}{\alpha}\right)^{\omm_k}(-1)^{k-1}\cdot \Theta_{\uM,k-1+\omm_{k-1}}=P_{k-1}(\aaa,\bbb)\cdot \Theta_{\uM,k-1+\omm_{k-1}}$,\\ where $P_{k}(\aaa,\bbb)=\left(\frac{\beta}{1-\alpha}\right)^{k+1}\left(\frac{\beta}{\alpha}\right)^{\omm_{k+1}}(-1)^{k}.$\\

Since $T^{k+\omm_{k}}(\beta)$ is bounded and  $P_{k}(\aaa,\bbb)\rightarrow +\infty$ for $(\aaa,\bbb)\in U$ as $k\to\oo$ we obtain that \\
\begin{equation}\label{041802a}
0=\lim_{k \rightarrow \infty}\Theta_{\uM,k+\omm_k}=\Theta(\alpha,\beta)=\Theta_{\underline{M}}(\alpha,\beta)=1-\beta+ \sum_{k=1}^{\infty}(-1)^{k}\left(\frac{1-\alpha}{\beta}\right)^{k}\left(\frac{\alpha}{\beta}\right)^{\omm_{k}}.
\end{equation}

\end{proof}

\begin{remark} \label{vegesthetam}
If $\uM\in \MM_{<,\oo}$, that is, $\uM=A_{0}...A_{n-1}C$ then
in the proof of the above theorem  we use $\uM^{-}$,
but in this case, as one can see by analyzing the proof of this theorem,
any sequence $A_{0}...A_{n-1}X_{1}A_{0}...A_{n-1}X_{2}A_{0}...A_{n-1}X_{3}...$
with $X_{1}X_{2}...\in \{ L,R \}^{\oo}$ would yield a function $\TTT_{\uM}$
satisfying the conditions of the theorem.
This is due to the fact that $T_{\aaa,\bbb}(\aaa)=L_{\aaa,\bbb}(\aaa)=R_{\aaa,\bbb}(\aaa)=\bbb$.

\end{remark}

\begin{lemma}\label{thprop}
Suppose $\underline{M} \in \mathfrak{M}\backslash\{RL^{\infty}\}, \; \frac{1}{2}<\beta<1$ is fixed. Then $  \Theta(\alpha)=  \Theta_{\underline{M}}(\alpha)=  \Theta_{\underline{M}}(\alpha,\beta)$ is analytic on $(1-\beta,\beta)$ and $\lim_{\alpha \rightarrow \beta-0}   \Theta(\alpha) = 0$.
\end{lemma}


\begin{proof}
First observe that for $r,s \geq0$, we have
\begin{equation}\label{02266a}
\left|\frac{ \partial}{ \partial \alpha} \left(\left( \frac{\alpha-1}{\beta}\right)^r\left(\frac{\alpha}{\beta}\right)^s\right)\right|=\left|\frac{r(\alpha-1)^{r-1}\alpha^s+s(\alpha-1)^r\alpha^{s-1}}{\beta^{r+s}}\right|
\end{equation}
$$\leq \frac{1}{\beta}(r+s) \max \Big \{\left(\frac{\alpha}{\beta}\right)^{r+s-1}, \left(\frac{1-\alpha}{\beta}\right)^{r+s-1}\Big \}
$$
moreover, by \eqref{02266a}
\begin{equation}\label{02266b}
\left|\frac{ \partial^2}{ \partial \alpha^{2}} \left(\left( \frac{\alpha-1}{\beta}\right)^r\left(\frac{\alpha}{\beta}\right)^s\right)\right| 
=
\left |  \frac{r}{\bbb}\dd_{\aaa} \left (\frac{(\aaa-1)^{r-1}\aaa^{s}}{\bbb^{r+s-1}} \right )+
\frac{s}{\bbb}\dd_{\aaa} \left (\frac{(\aaa-1)^{r}\aaa^{s-1}}{\bbb^{r+s-1}} \right )
\right |
\leq
\end{equation}
$$ \frac{1}{\beta^2}(r+s)(r+s-1)\max\Big \{\left(\frac{\alpha}{\beta}\right)^{r+s-2},\left(\frac{1-\alpha}{\beta}\right)^{r+s-2}\Big \}$$
and by induction, for any $j \leq r+s$ we have
\begin{equation}\label{02267a}
\begin{split}
\left|\frac{ \partial^j}{ \partial \alpha^{j}}\left(\frac{\alpha-1}{\beta}\right)^r\left(\frac{\alpha}{\beta}\right)^s\right| \leq & \frac{1}{\beta^j}(r+s)\dots  (r+s-j+1)\cdot \\
& \cdot \max \Big \{\left(\frac{\alpha}{\beta}\right)^{r+s-j},\left(\frac{1-\alpha}{\beta}\right)^{r+s-2}\Big \}.
\end{split}
\end{equation}
Since in the series $\eqref{*4b}$ the exponents $k+{\overline m}_k$ are strictly monotone increasing, that is $(k+1)+{\overline m}_{k+1}>k+{\overline m}_k$, using the above estimates one can easily see, that $  \Theta^{(j)}(\alpha)$ exists and equals the sum of the termwise differentiated series from \eqref{*4b} for $1-\beta < \alpha< \beta$. Therefore, $  \Theta(\alpha) \in \mathcal{C}^{\infty}(1-\beta,\beta).$
Suppose $\frac{1}{2} \leq \alpha <\beta$, which implies $\frac{\aaa}{\bbb}\geq \frac{1-\aaa}{\bbb}.$ Using \eqref{02261a} or \eqref{*4b} and \eqref{02267a} for $j \geq 1$ we obtain that
\begin{equation}\label{02268a}
\begin{split}
|  \Theta^{(j)}(\alpha)|&\leq \frac{1}{\beta^j}\sum_{k'=j}^{\infty}k'(k'-1)\dots(k'-j+1)\left(\frac{\alpha}{\beta}\right)^{k'-j} =\\
&=\frac{1}{\beta^j}j!\left(1-\frac{\alpha}{\beta}\right)^{-j-1}.
\end{split}
\end{equation}
Suppose $\alpha_0 \in (1-\beta,\beta)$. Without limiting generality we suppose $\frac{1}{2} \leq \alpha_0$, which implies $1-\alpha_0 \leq \alpha_0$. We denote by $\capp_n(\alpha)$ the $n$'th order Taylor polynomial of $  \Theta$ at the point $\alpha_0$. Then for $\alpha \in [\frac{1}{2},\beta)$ we have by $\eqref{02268a}$
\begin{equation}
\begin{split}
&|  \Theta(\alpha)-\capp_n(\alpha)| \leq \Big |\frac{  \Theta^{(n+1)}(\xi)}{(n+1)!}(\alpha-\alpha_0)^{n+1}\Big |\leq \\
&\leq \frac{1}{\beta^{n+1}} \cdot \frac{(n+1)!}{(n+1)!} \left(1-\frac{\xi}{\beta}\right)^{-n-2}|\alpha-\alpha_0|^{n+1}= \\
&=\frac{\beta}{\left(\beta-\xi\right)^{n+2}}|\alpha-\alpha_0|^{n+1},
\end{split}
\end{equation}
where $\xi$ is between $\alpha$ and $\alpha_0$. Hence,
 if $|\alpha-\alpha_0|<\frac{\beta-\alpha_0}{2}$ then $\beta-\xi>\frac{\beta-\alpha_0}{2}>|\alpha-\alpha_0|$ which implies that $\capp_n(\alpha) \rightarrow   \Theta(\alpha)$. Similar arguments can be used to the case when $\alpha_0 \leq \aaa \leq \frac{1}{2} $.
One could easily verify that the series in the definition of $  \Theta(\alpha)$ (see \eqref{*4b}) converges uniformly on $[\frac{1}{2},\beta]$ and \\ $\lim_{\alpha \rightarrow \beta-0}
  \Theta(\alpha)=1-\beta+\sum_{k=1}^{\infty}\left(\frac{\beta-1}{\beta}\right)^k=1-\beta+\frac{\beta-1}{\beta} \cdot
\frac{1}{1-\frac{\beta-1}{\beta}}=1-\beta+\frac{\beta-1}{\beta}\frac{\beta}{\beta-(\beta-1)}=0.$
We remark that  on  $[1-\bbb,\frac{1}{2}]$ close to $1-\bbb$ the absolute value of $\frac{1-\aaa}{\bbb}$ is close to one and hence the convergence properties of the series in the definition of $  \Theta(\alpha)$ can be worse close to $1-\bbb$.
Here we have only locally uniform convergence for example when  $\uM\in \MM_{R\oo}$. 
\end{proof}

\begin{proposition}\label{*minmon} Suppose $\uM,\tuM\in \MM\sm \{ RL^{\oo} \}$.
If $\uM\leqc \tuM$,  then $\uM^{-}\leqc \tuM^{-}$.  
\end{proposition}

\begin{proof}
If $\tuM$ is infinite then $\tuM^{-}=\tuM$ and we are done since $\uM^{-}\peq \uM.$

Suppose $\tuM=A_{0}...A_{n}A_{n+1}$, with $A_{n+1}=C$ and
$\uM^{-}=B_{0}B_{1}...$.

First suppose that $A_{0}...A_{n}$ is even. Then $\tuM^{-}=(A_{0}...A_{n}L)^{\oo}=A_{0}^{-}...A_{n}^{-}A_{n+1}^{-}...$.
Since $\uM^{-}\leqc\tuM$   there exists a least $j\leq n+1$  such that $B_{j}\not=A_{j}$. If $j\leq n$ then $A_{0}...A_{j-1}$ coincides with
the first $j$ entries of $\tuM^{-}$ and hence $\uM^{-}\leqc \tuM^{-}$. 

If $j=n+1$ then $B_{n+1}=L=A_{n+1}^{-}$.
Suppose that the $l$'th entry is the least where $B_{l}\not=A^{-}_{l}$ and $k(n+2)\leq l<(k+1)(n+2)$.
Then $B_{0}...B_{k(n+2)-1}=A_{0}^{-}...A^{-}_{k(n+2)-1}$ is even and $\sss^{k(n+2)}(\uM^{-})\peq \uM^{-}$ implies $$B_{k(n+2)}...B_{(k+1)(n+2)-1}\leqc A_{0}^{-}...A^{-}_{n+1}
=A_{k(n+2)}...A_{(k+1)(n+2)-1}=B_{0}...B_{n+1}$$ which implies $\uM^{-}\leqc \tuM^{-}.$

If $A_{0}...A_{n}$ is odd then $\tuM^{-}=(A_{0}...A_{n}R)^{\oo}=A_{0}^{-}...A_{n}^{-}A_{n+1}^{-}...$ and $A_{0}^{-}...A_{n}^{-}A_{n+1}^{-}$ will be again an even sequence and we can argue as earlier,
details are left to the reader.
\end{proof}

\begin{theorem}\label{felsozona}
Suppose $ { \underline { M } } \in \mathfrak{M}\sm\{RL^{\infty}\}$ is given, $(\alpha,\beta)\in U$ and $K(\alpha,\beta)=\widetilde{ { \underline { M } }} \succ  { \underline { M } }$. Then $  \Theta_{ { \underline { M } }}(\alpha,\beta)\neq 0.$
\end{theorem}


\begin{proof}
By maximality of $\uM$ and $ { \underline { M }^{-} }$ 
and by Proposition \ref{*minmon}
we have $\sigma^n  { \underline { M }^{-} }  \prec  { \underline { M }^{-} } \prec \widetilde{ { \underline { M }} }^{-}\peq \widetilde{\uM}$ for all $n \in \mathbb{N}$. 
This implies that there is $\tilde{x} \in [0,1]$, $\tilde{x} \neq \beta$ such that $\underline{I}_{\alpha,\beta}(\tilde{x})=\underline{I}(\tilde{X})= { \underline { M }^{-} }$.

  If $ { \underline { M }^{-} }=A_0A_1 \dots$ then $ { \underline { M }^{-} _{n}}=A_n \in \{R,L\}$, the $n$'th entry of the infinite string. 
  Recalling notation \eqref{*labx}
  this way $( { \underline { M }^{-} _n})_{\alpha,\beta}(x)$ is well defined and equals $L_{\alpha,\beta}(x)$ if $ { \underline { M }^{-} _n}=L$ and $R_{\alpha,\beta}(x)$ if $ { \underline { M }^{-} _n}=R$. \\

Then we can define $$T^{ { \underline { M }^{-} }|n}_{\alpha,\beta}(x)=( { \underline { M }^{-} _{n-1}})_{\alpha,\beta} \circ ( { \underline { M }^{-} _{n-2}})_{\alpha,\beta} \circ \dots \circ ( { \underline { M }^{-} _0})_{\alpha,\beta}(x).$$ 
For $\tilde{x}$ we have $T^{ { \underline { M }^{-} }|n}_{\alpha,\beta}(\tilde{x})=T^n_{\alpha,\beta}(\tilde{x})$, but
it is not necessarily true that
 $T^{ { \underline { M }^{-} }|n}_{\alpha,\beta}(\beta) = T^n_{\alpha,\beta}(\beta)$. We denote by $l( { \underline { M }^{-} }|n)$ the number of $L$'s and by $r( { \underline { M }^{-} }|n)$ the number of $R$'s in the string $ { \underline { M }^{-} } |n$. Since the functions $L_{\alpha,\beta}$ and $R_{\alpha,\beta}$ are linear we have
\begin{equation}\label{052103a}
|T^{ { \underline { M }^{-} }|n}_{\alpha,\beta}(\tilde{x})-T^{ { \underline { M }^{-} }|n}_{\alpha,\beta}(\beta)|= \left(\frac{\beta}{1-\alpha}\right)^{r( { \underline { M }^{-} }|n)}\left(\frac{\beta}{\alpha}\right)^{l( { \underline { M }^{-} }|n)}|\tilde{x}-\beta|.
\end{equation}
Now, suppose that $n=k+1+\omm_{k+1}$. Then
by the argument used in the proof of Theorem \ref{thx}
\begin{equation}\label{052104a}
T_{\alpha,\beta}^{ { \underline { M }^{-} }|n}(\beta)=P_k(\alpha,\beta)\cdot   \Theta_{ { \underline { M }^{-} },k+\omm_{k}}(\alpha,\beta), \text{ with }
\end{equation}
\begin{equation}\label{052104b}
\begin{split}
P_k(\alpha,\beta)&=\left(\frac{\beta}{1-\alpha}\right)^{k+1}\left(\frac{\beta}{\alpha}\right)^{\omm_{k+1}}(-1)^{k}=\\
&(-1)^{r( { \underline { M }^{-} }|n)+1} \cdot \left(\frac{\beta}{1-\alpha}\right)^{r( { \underline { M }^{-} }|n)}\cdot\left(\frac{\beta}{\alpha}\right)^{l( { \underline { M }^{-} }|n)} .
\end{split}
\end{equation}
Since $T^{ { \underline { M }^{-} }|n}_{\alpha,\beta}(\tilde{x}) \in [0,1]$ from \eqref{052103a} it follows that
\begin{equation}\label{052104c}
|T^{ { \underline { M }^{-} }|n}_{\alpha,\beta}(\beta)| \geq |P_k(\alpha,\beta)| \cdot |\tilde{x}-\beta|-1.
\end{equation}
Dividing \eqref{052104a} by $P_k(\alpha,\beta)$ and using \eqref{052104c} we obtain that
$$|  \Theta_{ { \underline { M } },k+\omm_{k}}(\alpha,\beta)| \geq |\tilde{x}-\beta|-\frac{1}{|P_k(\alpha,\beta)|}.$$

As $k$ goes to infinity we obtain that
$$|  \Theta_{ { \underline { M } }}(\alpha,\beta)| \geq |\tilde{x}-\beta|>0.$$
\end{proof}


This shows that  in the region $U$ above the curve $(\alpha,\Psi_{ { \underline { M } }}(\alpha)), \; \alpha \in (\alpha_1( { \underline { M } }), \alpha_2( { \underline { M } }))$ the auxiliary function $  \Theta_{ { \underline { M } }}$ is non-vanishing. For a while we tried to prove that a similar result is true for points of $U$ under the curve $(\alpha,\Psi_{ { \underline { M } }}(\alpha))$. The reason for this conjecture was that on all computer images we tried, for example $\alpha=0.5, \: \beta=0.7$, or $\alpha=0.5, \;\beta=0.8$  (see Figure \ref{theta78*}) this seemed to be the case (sometimes for some parameter values a small zone of the level zero set of $\TTT_{\uM}$ showed up close to $(1/2,1/2)$) but in this zone $\bbb/\aaa$
and $\bbb/(1-\aaa)$ are both close to one and hence we have slow convergence of the series in the definition of $\TTT$).
The next theorem shows that the level zero set of $\TTT_{\uM}(\aaa,\bbb)$ in $U$ does not necessarily coincides with the equi-kneading curve corresponding to $\uM$.


\begin{theorem}\label{thex}
There exists $ { \underline { M } } \in \mathfrak{M}_{\oo}\backslash \{RL^{\infty}\}$ such that for some $(\alpha,\beta) \in U$ we have $  \Theta_{ { \underline { M } }}(\alpha,\beta)=0$, but $K(\alpha,\beta)=\widetilde{ { \underline { M } }} \neq  { \underline { M } }$. By Theorem \ref{felsozona} we have $K(\alpha,\beta)  \prec  { \underline { M } }$.
\end{theorem}


\begin{proof}
We use notation from the proof of Theorem \ref{felsozona}. Suppose that 
 $$\underline{M}=\underline{M}^{-}=R\underbrace{L\dots L}_{m_1}R\underbrace{L\dots L}_{m_2}R\underbrace{L\dots L}_{m_3}R \dots$$
is defined by
$m_1=6,$  $m_{2k}=5$, $m_{2k+1}=0$ for $k=1,...,23$ and $m_{k}=0$ for $k \geq 48$. This uniquely defines $ { \underline { M } }$ and it is also clear that $ { \underline { M } }$ ends with an infinite string of $R$'s. Moreover, one can easily see that $ { \underline { M } }$ is a maximal sequence and $ { \underline { M } } \neq RL^{\infty}$. In order to verify that $ { \underline { M } } \in \mathfrak{M}$ we show that one cannot write $ { \underline { M } }$ in the form $\underline{A} * \underline{B}$. Since $\underline{A} * \underline{B}$ contains infinitely often the finite string $\underline{A}$ and $ { \underline { M } }$ ends with $R^{\infty}$ the string $\underline{A}$ can contain only $R's$. The blocks $\underline{A}$ are separated by one symbol and $ { \underline { M } }$ is of the form $\underline{A}\star_1\underline{A}\star_2\underline{A}\star_3\underline{A}\star\dots$ where the symbols $\star_k$ are determined by $\underline{B}$. Since $\lambda_1=6>0$ we should have $\underline{A}=R$ and then $ { \underline { M } }$ should be of the form $\underline{R}\star_1\underline{R}\star_2\underline{R}\star_3\dots$ contradicting the fact that the third symbol in $ { \underline { M } }$ is an $L$, not an $R$. 

Therefore, $ { \underline { M } } \in \mathfrak{M}_{R,\infty} \subset \mathfrak{M}_{\oo}\backslash\{RL^{\infty}\}$. After some computer 
experimentation we chose $\alpha_0=\alpha_1=\alpha_2=0.4875, \: \beta_0=0.535, \: \beta_1=0.7$ and $\beta_2=0.995.$ Out of these values the choice of $(\alpha_0,\beta_0)$ was difficult. One can write $  \Theta_{ { \underline { M } }}$ in closed form:
$$  \Theta_{ { \underline { M } }}(\alpha,\beta)=1-\beta+\left(\frac{\alpha-1}{\beta}\right)\left(\frac{\alpha}{\beta}\right)^6+
$$ $$\sum_{k=1}^{23}
\left(\left(\frac{\alpha-1}{\beta}\right)^{2k}\left(\frac{\alpha}{\beta}\right)^{6+5k}+\left(\frac{\alpha-1}{\beta}\right)^{2k+1}\left(\frac{\alpha}{\beta}\right)^{6+5k}\right)+\sum_{k=48}^{\infty} \left(\frac{\alpha-1}{\beta}\right)^k\left(\frac{\alpha}{\beta}\right)^{6+5 \cdot 23}=$$ $$1-\beta+\frac{\alpha-1}{\beta}\left(\frac{\alpha}{\beta}\right)^6+\sum_{k=1}^{23}\left(\left(\frac{\alpha-1}{\beta}\right)^{2k}\left(\frac{ \alpha}{\beta}\right)^{6+5k}\cdot
 \left(\frac{\alpha+\beta-1}{\beta}\right)\right)+
 $$ $$
 \left(\frac{\alpha}{\beta}\right)^{121} \cdot \left(\frac{\alpha-1}{\beta}\right)^{48}\frac{\beta}{1+\beta-\alpha}.$$
With the above values
$$  \Theta_{ { \underline { M } }}(\alpha_0,\beta_0)\app-0.0505893,$$
$$  \Theta_{ { \underline { M } }}(\alpha_1,\beta_1)\app0.2194096,$$
$$  \Theta_{ { \underline { M } }}(\alpha_2,\beta_2)\app-0.00207430.$$
The error of our double precision calculation was
 less than $0.0001$ which implies that the continuous function $h(t)=  \Theta_{ { \underline { M } }}(\alpha_0,t)$ has
at least two roots between $\beta_0$ and $\beta_2$. On the other hand, by the strict monotonicity of $h(\alpha_0,t)$ it is impossible that for both of these roots $K(\alpha_0,t)= { \underline { M } }$. Hence, there is $(\alpha,\beta)$ with $\aaa=\alpha_0$
and $\beta\in (\bbb_{0},\bbb_{1})$ satisfying the claim of the theorem.
\end{proof}
\begin{center}
\begin{figure}
\includegraphics{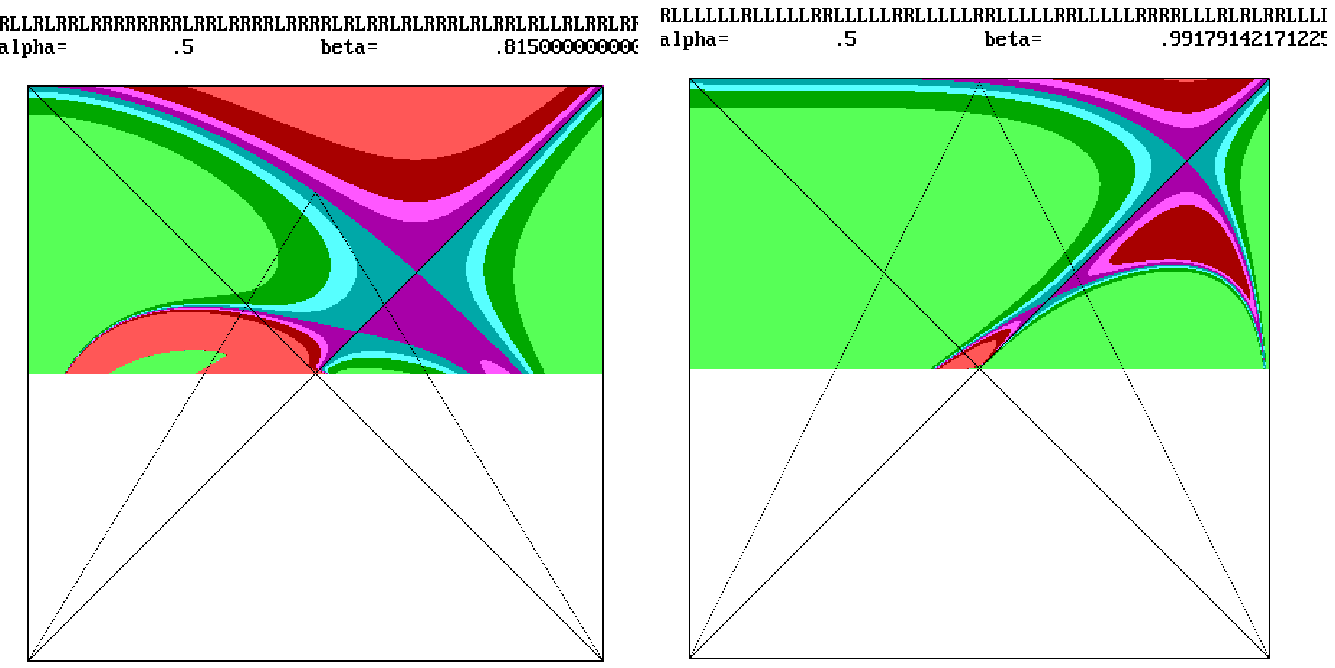}
\caption{Level set structure of $\TTT$ when $\uM=K(.5,.815)$ and when $\uM=K(.5,.8)$} 
\label{theta34*}
\end{figure}
\end{center}
The kneading sequence $ { \underline { M } }$ in Theorem \ref{thex} was defined by some theoretical considerations. To obtain a maximal sequence $m_1=6$
had to be larger than the other $m_j$'s. The key negative term $\left(\frac{\alpha_0-1}{\beta_0}\right)\cdot\left(\frac{\alpha_0}{\beta_0}\right)^6\app -.5483592$ had to be of larger absolute value than $1-\beta_0=0.465$, and the $23$ positive terms of the finite sum were made small by the fact that this sum can also be written in the form $$\left(\frac{\alpha_0-1}{\beta_0}\right)^2\left(\frac{\alpha_0}{\beta_0}\right)^{11} \cdot \left(\frac{\alpha_0+\beta_0-1}{\beta_0}\right) \cdot \sum_{k=0}^{22} \left(\frac{\alpha_0}{\beta_0}\right)^{5k}$$ and
$$\left(\frac{\alpha_0-1}{\beta_0}\right)^2\left(\frac{\alpha_0}{\beta_0}\right)^{11}\left(\frac{\alpha_0+\beta_0-1}{\beta_0}\right)\approx 0.0138784$$
with $\frac{\alpha_0+\beta_0-1}{\beta_0} \approx 0.0420561$. The reciprocal of this 
last number is quite large, but by the choice of $23$ terms $\left(\frac{\alpha_0}{\beta_0}\right)^{121}\left(\frac{\alpha_0-1}{\beta_0}\right)^{48} \frac{\beta_0}{1+\beta_0-\alpha_0} \approx 8.445889\cdot 10^{-07}$ is small again.

We remark that by keeping $\alpha_0=0.5$ fixed and letting $\beta_0$ change between $0.5$ and $1$ we made a plot of the level set structure
of $  \Theta_{K(\alpha_{0},\beta_{0})}$ for a hundred different values of $\beta_0$ and found an interval where for several parameter values, like $0.815$ 
(see left side of figure \ref{theta34*})
there is a substantially large region $\Omega \subset U$ such that $  \Theta_{K(\alpha_0,\beta_0)} <0$ on $\Omega, \: (\frac{1}{2}, \frac{1}{2})$ is on the boundary of $\Omega$, and points $(\alpha,\beta) \in U$ on the boundary of $\Omega$ satisfy $  \Theta_{K(\alpha_0,\beta_0)}(\alpha,\beta)=0$ but
$K(\alpha,\beta) \neq K(\alpha_0,\beta_0)$. We remark that the points $(\alpha,\beta) \in U$ with $K(\alpha,\beta) \succ  { \underline { M } }=K(\alpha_0,\beta_0)$, that is points mentioned in Theorem \ref{felsozona} also form a region where $  \Theta_{K(\alpha_0,\beta_0)}<0$, but
the boundary of this region contains points where $K(\alpha,\beta)= { \underline { M } }=K(\alpha_0,\beta_0)$. For most parameter values under the equi-kneading curve $K(\alpha,\beta)= { \underline { M } }$ we have a positive region, where $  \Theta_{K(\alpha_0,\beta_0)}>0$, but as the above examples show this is not always the case.

Another critical/exceptional parameter value was $\bbb_{0}=0.99179142171225$, the level set structure corresponding to this one is illustrated on the right side of Figure \ref{theta34*}.


\section{Further properties of $\TTT_{\uM}$}\label{secfurTTTM}

\begin{lemma}\label{thdiff}
For any $ { \underline { M } } \in \mathfrak{M}\backslash \{RL^{\infty}\}$ the function $  \Theta_{ { \underline { M } }}(\alpha,\beta)$ is infinitely differentiable on $U$ and also in small neighborhoods of the points $\{(\bbb,\beta): \frac{1}{2}<\beta<1\}$ on the boundary of $U$. Its first partial derivatives are
\begin{equation}\label{030901a}
 \partial_{\alpha}  \Theta_{ { \underline { M } }}(\alpha,\beta)=\sum_{k=1}^{\infty}\left(\frac{k}{\beta}\left(\frac{\alpha-1}{\beta}\right)^{k-1}\left(\frac{\alpha}{\beta}\right)^{{\overline m}_k}+\frac{{\overline m}_k}{\beta}\left(\frac{\alpha-1}{\beta}\right)^k\left(\frac{\alpha}{\beta}\right)^{{\overline m}_k-1}\right)
\end{equation}
and
\begin{equation}\label{030901b} \partial_{\beta}  \Theta_{ { \underline { M } }}(\alpha,\beta)=-1+\sum_{k=1}^{\infty}\left(\frac{-k-{\overline m}_k}{\beta}\right)\left(\frac{\alpha-1}{\beta}\right)^{k}\left(\frac{\alpha}{\beta}\right)^{{\overline m}_k},
\end{equation}
its second partials are
\begin{equation}\label{030901c}
\begin{split}
 \partial^2_{\alpha}  &\Theta_{ { \underline { M } }}(\alpha,\beta)= \sum_{k=1}^{\infty}\left(\frac{k(k-1)}{\beta^2}\left(\frac{\alpha-1}{\beta}\right)^{k-2}\left(\frac{\alpha}{\beta}\right)^{{\overline m}_k}+\right.\\ &\left.
2\frac{k \cdot {\overline m}_k}{\beta^2}\left(\frac{\alpha-1}{\beta}\right)^{k-1}\left(\frac{\alpha}{\beta}\right)^{{\overline m}_k-1}+\frac{{\overline m}_k({\overline m}_k-1)}{\beta^2}\left(\frac{\alpha-1}{\beta}\right)^{k}\left(\frac{\alpha}{\beta}\right)^{{\overline m}_k-2}\right),
\end{split}
\end{equation}
\begin{equation}\label{030902a}
\begin{split}
 \partial_{\alpha}  \partial_{\beta}  \Theta_{ { \underline { M } }}(\alpha,\beta)= &\partial_{\beta}  \partial_{\alpha}  \Theta_{ { \underline { M } }}(\alpha,\beta)= \sum_{k=1}^{\infty}\left(\frac{k(-k-\omm_{k})}{\beta^2}\left(\frac{\alpha-1}{\beta}\right)^{k-1}\left(\frac{\alpha}{\beta}\right)^{{\overline m}_k}+\right.\\  &\left.\frac{{\overline m}_k (-k-{\overline m}_k)}{\beta^{2} }\left(\frac{\alpha-1}{\beta}\right)^{k}\left(\frac{\alpha}{\beta}\right)^{{\overline m}_k-1}\right),
\end{split}
\end{equation}
\begin{equation}\label{030902b}
\begin{split}
& \partial^2_{\beta}  \Theta_{ { \underline { M } }}(\alpha,\beta)=
\sum_{k=1}^{\infty}\left(\frac{(k+{\overline m}_k)(k+\omm_{k}+1)}{\beta^2}\left(\frac{\alpha-1}{\beta}\right)^{k}\left(\frac{\alpha}{\beta}\right)^{{\overline m}_k} \right ).
\end{split}
\end{equation}
\end{lemma}


\begin{proof}
The series in $\eqref{030901a}$ and $\eqref{030901b}$ are obtained by termwise partial differentiation of the series definition of
$  \Theta_{ { \underline { M } }}(\alpha,\beta)$ given in $\eqref{*4b}$. If $1-\beta<\alpha<\beta$ then $\left|\frac{\alpha-1}{\beta}\right|<1$ and
$\left|\frac{\alpha}{\beta}\right|<1$. Hence the convergence of the series in $\eqref{030901a}$-$\eqref{030902b}$ is locally uniform in $U$. 
From Theorem \ref{thx} it follows that
\begin{equation}\label{030903a}
{\overline m}_k \leq k{\overline m}_1.
\end{equation}
 One can choose $\eta_{\beta} \in (0,1)$ and a small neighborhood $V_{\beta}$ of a point $(\beta,\beta), \; \frac{1}{2} < \beta <1 $ such that
\begin{equation}\label{030903b}
\left|\left(\frac{\alpha-1}{\beta}\right) \left(\frac{\alpha}{\beta}\right)^{{\overline m}_1}\right|<\eta_{\beta}<1.
\end{equation}
holds for all
$(\alpha,\beta) \in V_{\beta}$. In this neighborhood for parameter values under the diagonal $|\aaa/\bbb|>1$, but it is sufficiently close to $1$. Using $\eqref{030903a}$ one can easily verify the uniform convergence on $V_{\beta}$ of the series in $\eqref{030901a}$ and $\eqref{030901b}$. Similar arguments based on the estimates $\eqref{030903a}$ and $\eqref{030903b}$ can be used to verify locally uniform convergence of the series in $\eqref{030901c}$, $\eqref{030902a}$ and $\eqref{030902b}$. This shows that $  \Theta_{ { \underline { M } }}$ is twice continuously differentiable at the points considered in our theorem. For the higher derivatives  using $\eqref{030903a}$ and $\eqref{030903b}$ 
one can show similarly that the termwise $ \partial_{\beta}^j  \partial_{\alpha}^i$ partial derivative series in $\eqref{*4b}$ converges uniformly in $V_{\bbb}.$
\end{proof}


Next we need an estimate of ${\overline m}_1$ appearing in the definition of $  \Theta_{ { \underline { M } }}(\alpha,\beta)$ in $\eqref{02261a}$ or \eqref{*4b}.
Suppose
\begin{equation}\label{030906a}
\uM\in \mathfrak{M}\backslash \{RL^{\infty}\}\text{ and } { \underline { M }|(m_{1}+2} )=R \underbrace{L \dots L}_{m_{1}}X , \text{ with $X \in \{R,C\}$.}
\end{equation}

Then,  $ m_1={\overline m}_1$. This means that we need to estimate $m_1=\omm_{1}(\aaa,\bbb)$ when $ { \underline { M } }=K(\alpha,\beta)$ with $\alpha$ close to $\beta$.


\begin{lemma}\label{mest}
Assume $\beta_0 \in (\frac{1}{2},1)$ is fixed. Then there exists a neighborhood $V_{\beta_0}$ of $(\beta_0,\beta_0)$ such that for all $(\alpha,\beta) \in V_{\beta_0} \cap U$ we have
\begin{equation}\label{030907a}
\frac{\beta_0}{1-\beta_0} -2 < {\overline m}_1(\alpha,\beta)<\frac{\beta_0}{1-\beta_0}+1
\end{equation}
\end{lemma}

\begin{proof}
Choose a neighborhood $V_{\beta_0}$ of $(\beta_0,\beta_0)$ such that for all $(\alpha,\beta) \in V_{\beta_0}\cap U$.
We have
\begin{equation}\label{030907b}
\max\left(\left|\frac{\aaa}{1-\beta} -\frac{\beta_0}{1-\beta_0}\right|,\left|\frac{\bbb}{1-\aaa} -\frac{\beta_0}{1-\beta_0}\right|\right)<\frac{1}{100}.
\end{equation}
The value ${\overline m}_1={\overline m}_1(\alpha,\beta)$ is determined by the property
\begin{equation}\label{030908a}
\begin{split}
T^{1+\omm_{1}}\saab (\bbb)=&\frac{\beta}{1-\alpha}(1-\beta) \left(\frac{\beta}{\alpha}\right)^{{\overline m}_1} \geq \alpha\\
\text{but } T^{\omm_{1}}\saab (\bbb)=&\frac{\beta}{1-\alpha}(1-\beta) \left(\frac{\beta}{\alpha}\right)^{{\overline m}_1-1} < \alpha.
\end{split}
\end{equation}
This implies
\begin{equation}\label{030908b}
\frac{\log(1-\alpha)-\log(1-\beta)}{\log(\beta)-\log(\alpha)}-1 \leq {\overline m}_1 < \frac{\log(1-\alpha)-\log(1-\beta)}{\log(\beta)-\log(\alpha)}.
\end{equation}
By Cauchy's mean value theorem there exists $\gamma \in (\alpha,\beta)$ such that
\begin{equation}\label{030908c}
\frac{\log(1-\alpha)-\log(1-\beta)}{\log(\beta)-\log(\alpha)}=\frac{-\frac{1}{1-\gamma}}{-\frac{1}{\gamma}}=\frac{\gamma}{1-\gamma}.
\end{equation}
Since $\aaa<\ggg<\bbb$
\begin{equation}\label{030909a}
\frac{\aaa}{1-\bbb}<\frac{\ggg}{1-\ggg}<\frac{\bbb}{1-\aaa}
\end{equation}
From \eqref{030907b}, \eqref{030908b} and \eqref{030909a} we infer \eqref{030907a}.
\end{proof}


\section{The almost orthogonality of the equi-kneading curves to the diagonal}\label{secortho}

\begin{figure}
\begin{center}
\includegraphics[width=7.0cm]{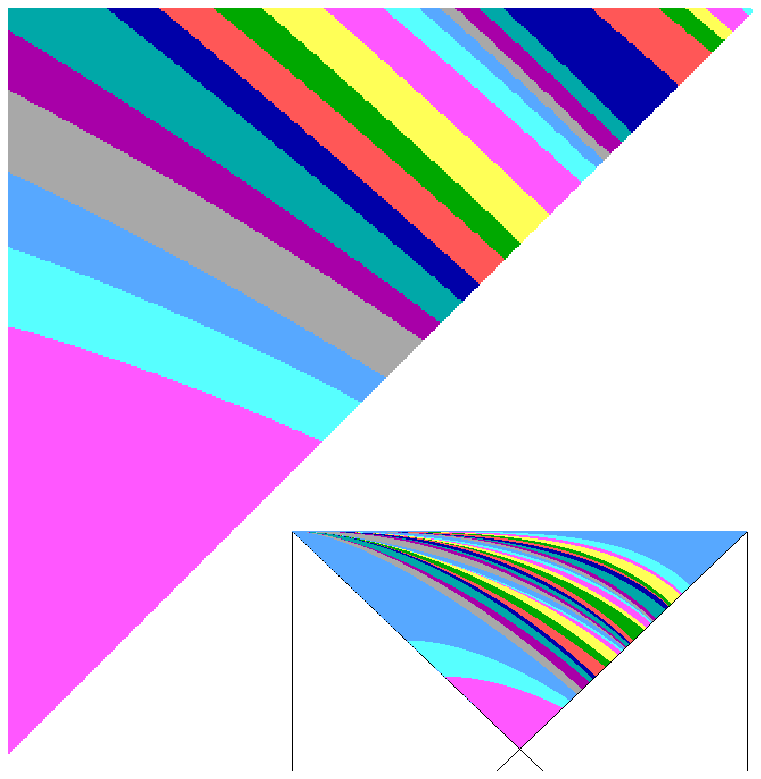}
\caption{Blow up of equi-kneading curves near $(1/2,1/2)$} 
\label{fig505*}
\end{center}
\end{figure}

By Theorem \ref{MV}, $\lim_{\mu\to\GGG(\uM)-}\fff_{\uM}(\mmm)=\fff_{\uM}(\GGG(\uM))=1$. Recall that the transformation $(\lll(\aaa,\bbb),\mmm(\aaa,\bbb))=(\frac{\bbb}{\aaa},\frac{\bbb}{1-\aaa})$ maps $U$ onto $D$ and
$\lll=1$ means $\aaa=\bbb$. Hence $\lim_{\aaa\to\aaa_{2}(\uM)-}\PPP_{\uM}(\aaa)=\bbb_{\uM}=\aaa_{2}(\uM)$ exists and $(\frac{\bbb_{\uM}}{\bbb_{\uM}},\frac{\bbb_{\uM}}{1-\bbb_{\uM}})=(1,\GGG(\uM))$.
That is, the curves
$(\alpha,\Psi_{ { \underline { M } }}\left(\alpha)\right), \; \alpha \in (\alpha_1( { \underline { M } }),\alpha_2( { \underline { M } }))$ reach the boundary of $U$ at  the point $(\bbb_{\uM},\bbb_{\uM})$ on the line segment
$\{(\beta,\beta):\frac{1}{2}<\beta<1\}$.
 Reversing our notation for $\beta_0 \in \left(\frac{1}{2},1\right)$ now we denote by $\Psi^{\beta_0}$ the curve
with the property $\lim_{\alpha \rightarrow \beta_0-} \Psi^{\beta_0}(\alpha)=\beta_0.$ In fact, we can extend the definition of $\Psi^{\beta}$
setting $\Psi^{\beta_0}(\beta_0)=\beta_0$. Moreover, by Theorem \ref{thx}, $  \Theta_{ { \underline { M } }_{\beta_0}}(\alpha,
\Psi^{\beta_0}(\alpha))=  \Theta^{\beta_0}(\alpha, \Psi^{\beta_0}(\alpha))=0$ for $(\alpha,\Psi^{\beta_0}(\alpha)) \in U$. We have seen in Lemma
\ref{thdiff} that $\TTT^{\beta_0}(\alpha,\beta)$ is infinitely differentiable at the point $(\beta_0,\beta_0)$ as well. 

Therefore, using the
implicit definition one could even extend the
definition of $\Psi^{\beta_0}$ onto a small interval $(\beta_0, \beta_0+ { \varepsilon })$, by setting $  \Theta^{\beta_0}(\alpha,\Psi^{\beta_0}(\alpha))=0$. Indeed, it is possible but a little caution is necessary.
\begin{lemma}\label{thb0}
Suppose $\beta_0 \in (\frac{1}{2},1)$ and we consider $\Psi^{\beta_0}, \:   \Theta^{\beta_0}$ defined as above. Then
\begin{equation}\label{030911a}
  \Theta^{\beta_0}(\beta,\beta)=0
\end{equation}
for all $\beta \in (\frac{1}{2},1)$
\begin{equation}\label{030911c}
 \partial_{\alpha}   \Theta^{\beta_0}(\beta_0,\beta_0)=0, \; \partial_{\beta}   \Theta^{\beta_0}(\beta_0,\beta_0)=0.
\end{equation}
\end{lemma}


\begin{proof}
Obviously $  \Theta^{\beta_0}(\beta,\beta)=1-\beta+\sum_{k=1}^{\infty}\left(\frac{\beta-1}{\beta}\right)^{k}=0$. \\ 

If we had
 $( \partial_{\alpha}   \Theta^{\beta_0}(\beta_0,\beta_0), \partial_{\beta}\Theta^{\beta_0}(\beta_0,\beta_0)) \neq (0,0),$ then 
 by the Implicit Function Theorem the level set $  \Theta^{\beta_0}(\alpha,\beta)=   \Theta^{\beta_0}(\alpha_0,\beta_0)=0$ would equal a differentiable curve in a small neighborhood of $(\beta_0,\beta_0)$. We have two curves $\{(\beta,\beta):\beta \in (\frac{1}{2},1)\}$ and the graph of $(\alpha,\Psi^{\beta_0}(\alpha))$ which arrives to $(\beta_0,\beta_0)$ from $U$ and $  \Theta^{\beta_0}$ vanishes on both,
which is impossible.
This implies $\eqref{030911c}$.
\end{proof}


Hence we cannot determine the derivative of $\Psi^{\beta_0}$ at $\beta_0$ by using implicit differentiation of $  \Theta^{\beta_0}(\alpha,\Psi^{\beta_0}(\alpha))$ at $\alpha=\beta_0$. To avoid awkward notation we denote by $D_{\alpha} \Psi^{\beta_0}$ the derivative of $\Psi^{\beta_0}$.
\begin{theorem}\label{thperp}
With the above notation
\begin{equation}\label{030913a}
\lim_{\beta_0 \rightarrow 1-} D_{\alpha} \Psi^{\beta_0}(\beta_0)=-1.
\end{equation}
This means that the curves $\Psi^{\beta_0}$ are almost perpendicular to the diagonal $\{(\beta,\beta):\frac{1}{2}<\beta<1\}$ at the point $(\beta_0,\beta_0)$ when $\beta_0$ is close to $1$.
\end{theorem}


Before proving this theorem we give some examples.
M. Misiurewicz asked the first listed author at a conference whether the formula
$\lim_{\beta_0 \rightarrow 1-} D_{\alpha} \Psi^{\beta_0}(\beta_0)=-1$ can be improved to $ D_{\alpha} \Psi^{\beta_0}(\beta_0)=-1.$ Of course, especially for $ { \beta }_{0}$ close to $1/2$ the computer images  (see Figure \ref{fig505*}) suggest that this is not the case.
One can look at some specific kneading sequences,
say corresponding to periodic turning points,
and consider the implicitely defined algebraic curves.

{\bf The first, easiest attempt, consider ${ \underline { M } }=RLC$.}

Recall that
$$R_{ { \alpha }, { \beta }}(x)={\beta\over\alpha-1}(x-1)\text{ and }
L_{ { \alpha }, { \beta }}(x)={\beta\over\alpha}x.$$
and we need to consider
$$T\saab^{2}(\bbb)=L_{ { \alpha }, { \beta }}(R_{ { \alpha }, { \beta }}(\beta))=L_{ { \alpha }, { \beta }}(R_{ { \alpha }, { \beta }}(L_{ { \alpha }, { \beta }}( { \alpha })))= { \alpha }.$$
This yields after substitution and rearrangement the implicit equation
$$-{\alpha ^3-\alpha ^2-\beta ^3+\beta^2\over \alpha  (\alpha -1)}=0.$$
Taking the numerator
we obtain $\capp( { \alpha }, { \beta })=\alpha ^3-\alpha ^2-\beta ^3+ \beta^2=0.$
By using the solution formula for cubic equations
one can express from this the implicit function
$ { \beta }( { \alpha })$, satisfying $\capp( { \alpha }, { \beta }( { \alpha }))=0$.
We will follow a different approach which will be used later for
the case when the degree of $\capp( { \alpha }, { \beta })$ is five
and one cannot use a solution formula.
We need to find the point $( { \beta }_{0}, { \beta }_{0})$ which is
an intersection point of the graph of our implicit curve $( { \alpha }, { \beta }( { \alpha }))$
 and the diagonal $( { \beta }, { \beta })$. Since  $\capp( { \beta }, { \beta })=0$,
that is, $\capp$ vanishes on the diagonal
at the point $( { \beta }_{0}, { \beta }_{0})$ the two implicit curves defined
by $\capp( { \alpha }, { \beta })=0$ intersect each other at this point. If the intersection angle is nonzero it is possible only if the first derivative of $\capp$ vanishes at
$( { \beta }_{0}, { \beta }_{0})$.

If we take
$$\frac{ {  \partial } \capp( { \alpha }, { \beta })}
{ {  \partial }  { \alpha }}= 3 \alpha ^2-2 \alpha
.$$
Substituting $ { \alpha }= { \beta }$, for seeking points on the diagonal,
we obtain that $
\alpha ={2\over 3}.
$
Since the first differential of $\capp( { \alpha }, { \beta })$ vanishes at $( { \beta }_{0}, { \beta }_{0})$ one cannot use implicit differentiation to determine $ { \beta }'( { \beta }_{0})$.
 The second differential $d^{2}\capp(( { \alpha }, { \beta }),(x,y))=(6 \alpha -2
)x^2+2\cdot 0\cdot xy+
(2-6 \beta
)y^2.$
At $( { \beta }_{0}, { \beta }_{0})$ this reduces to
$$d^{2}\capp(( { \beta }_{0}, { \beta }_{0}),(x,y))=2x^2-2y^2.$$
If we divide by $x^{2}$ and introduce the new variable $z=y/x$
then we obtain that the equation $d^{2}\capp(( { \beta }_{0}, { \beta }_{0}),(1,z))=0$ has two roots $z_{1}=-1$ and $z_{2}=1$.
This gives the slopes of the two intersecting implicit curves at
$( { \beta }_{0}, { \beta }_{0})$, the first slope $-1$ equals $ { \beta }'( { \beta }_{0})$,
while the second slope $1$ is not the least surprizing, since this is the slope
of the diagonal, which is also an implicit "curve" defined by $\capp( { \alpha }, { \beta })=0$.
The above calculation is quite simple, and shows that
the curve $\Psi^{\beta_0}( { \beta })$ corresponding to $ { \beta }_{0}=2/3$
is perpendicular to the diagonal, that is, $D_{\alpha} \Psi^{\beta_0}(\beta_0)=-1$ and this is not a good example to answer M. Misiurewicz's question.

After some experimentation the kneading sequence $RLLRC$
looked more promising.

{\bf The second attempt, consider ${ \underline { M } }=RLLRC$.}

Now
$R_{ { \alpha }, { \beta }}(L_{ { \alpha }, { \beta }}(L_{ { \alpha }, { \beta }}(R_{ { \alpha }, { \beta }}(\beta))))=R_{ { \alpha }, { \beta }}(L_{ { \alpha }, { \beta }}(L_{ { \alpha }, { \beta }}(R_{ { \alpha }, { \beta }}(L_{ { \alpha }, { \beta }}( { \alpha }))))= { \alpha }$ leads to the implicit equation
$$-{\alpha^5-2\alpha^4+\alpha^3(\beta+1)-\alpha^2\beta-\beta^4(\beta-1)\over \alpha^2(\alpha-1)^2}=0.$$
\begin{center}
\begin{figure}
\includegraphics{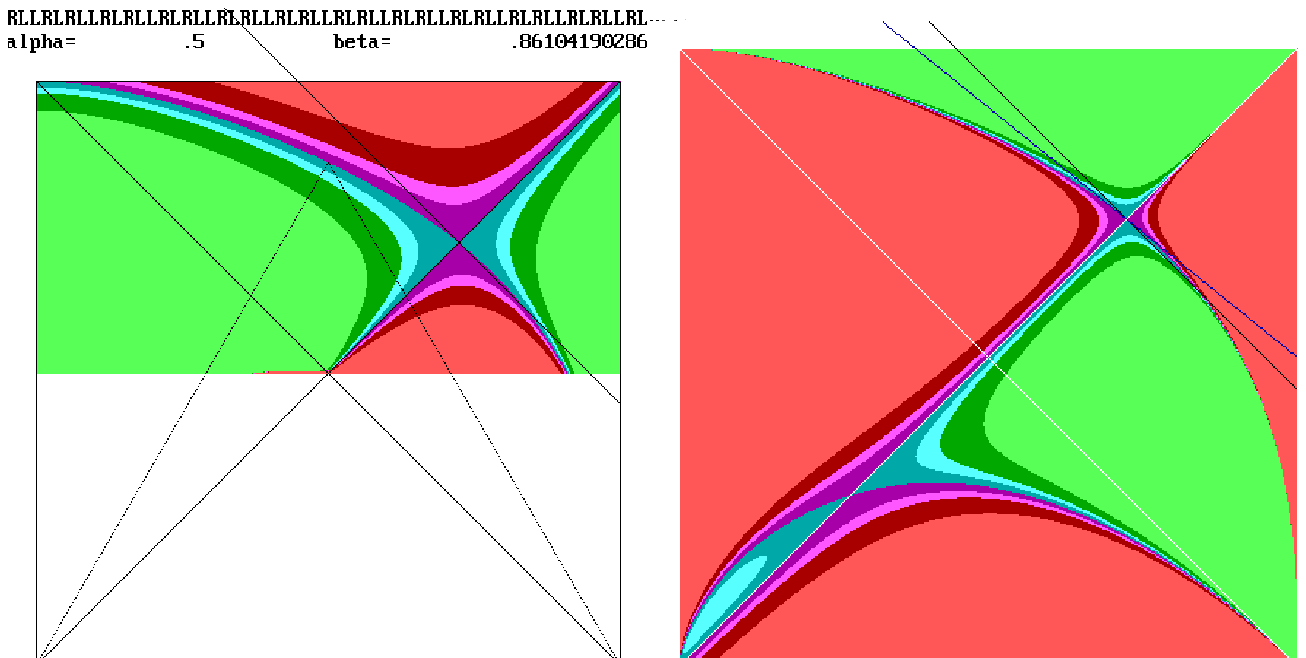}
\caption{Level set structure of $\TTT_{RLLRC}$ and $\capp( { \alpha }, { \beta })$} 
\label{thetarllrc*}
\end{figure}
\end{center}
Taking the numerator
we obtain $\capp( { \alpha }, { \beta })=\alpha^5-2\alpha^4+\alpha^3(\beta+1)-\alpha^2\beta-\beta^4(\beta-1)=0.$
Since at the point $( { \beta }_{0}, { \beta }_{0})$ the first differential
of $\capp( { \alpha }, { \beta })$ vanishes we can consider
$$\frac{ {  \partial } \capp( { \alpha }, { \beta })}
{ {  \partial }  { \alpha }}=5\alpha^4-8\alpha^3+3\alpha^2(\beta+1)-2\alpha\beta.$$
Substituting $ { \alpha }= { \beta }$, for seeking points on the diagonal,
we obtain the equation $$5\alpha^4-8\alpha^3+3\alpha^2(\alpha+1)-2\alpha^2=\alpha^2(5\alpha^2-5\alpha+1)
=0.$$
This equation has two non-zero roots
$$\alpha_{0}={2\over \sqrt{5}+5} \text{ and }
\alpha_{1}={\sqrt{5}\over10}+{1\over2}.
$$
Since $ { \alpha }_{0}<1/2$ we obtain that
$$ { \beta }_{0}=\alpha_{1}={\sqrt{5}\over10}+{1\over2}.
$$
Since the first differential of $\capp( { \alpha }, { \beta })$ vanishes at $( { \beta }_{0}, { \beta }_{0})$ one cannot use implicit differentiation to determine $ { \beta }'( { \beta }_{0})$. The second differential $d^{2}\capp(( { \alpha }, { \beta }),(x,y))=
(20 \alpha ^3-24 \alpha ^2+6 \alpha  (\beta +1)-2 \beta)x^{2}+
2(3 \alpha ^2-2 \alpha)xy+ (-4 \beta ^2 (5 \beta -3))y^{2}.$
At $( { \beta }_{0}, { \beta }_{0})$ this reduces to
$$d^{2}\capp(( { \beta }_{0}, { \beta }_{0}),(x,y))=\frac{1}{2}\Big ({\sqrt{5}\over 5}+1\Big )^2 x^{2}+
\frac{1}{2}\Big ({3 \sqrt{5}\over 5}-1\Big ) \Big ({\sqrt{5}\over 5}+1\Big )xy
-\frac{1}{2}(\sqrt{5}-1) \Big ({\sqrt{5}\over 5}+1\Big )^2y^2.$$
If we divide by $x^{2}$ and introduce the new variable $z=y/x$
then we obtain that the equation $d^{2}\capp(( { \beta }_{0}, { \beta }_{0}),(1,z))=0$ has two roots $z_{1}=
-{\sqrt{5}+3\over 2 \sqrt{5}+2}
 \approx -0.80901699437495$ and $z_{2}=1$.
The first root equals $ { \beta }'( { \beta }_{0})\not=-1$.
On the left side of Figure \ref{thetarllrc*} one can see the level set structure of $\ds \TTT_{RLLRC}$ and on the right side that of the 
polynomial $\capp( { \alpha }, { \beta })$. On the right side we also plotted the tangent line with the slope $-{\sqrt{5}+3\over 2 \sqrt{5}+2}$ and the line with slope $-1$, that is the one which is perpendicular to the diagonal. Observe that $RLLRC^-=(RLLRL)^{\oo}$ appears on the figure as the kneading sequence used by the computer.

\begin{proof}[Proof of Theorem \ref{thperp}.]
By Lemma \ref{thb0} we cannot use implicit differentiation to prove our theorem, we need to consider again the second differential instead. By \eqref{030901c}
\begin{equation}\label{030913b}
\begin{split}
\partial^2_{\alpha}   \Theta^{\beta_0}&(\beta_0,\beta_0)=\frac{2{\overline m}_1}{\beta_0^2}+\frac{{\overline m}_1({\overline m}_1-1)}{\beta_0^2} \cdot \left (\frac{\beta_0-1}{\beta_0} \right )+ \\
&+\sum_{k=2}^{\infty} \left(\frac{k(k-1)}{\beta_0^2}\left(\frac{\beta_0-1}{\beta_0}\right)^{k-2}+2\frac{k\cdot  {\overline m}_k}{\beta_0^2}\left(\frac{\beta_0-1}{\beta_0}\right)^{k-1}+\frac{{\overline m}_k({\overline m}_k-1)}{\beta_0^2}\left(\frac{\beta_0-1}{\beta_0}\right)^k\right).
\end{split}
\end{equation}
One can again use the estimates  $\eqref{030903a}$ and $\eqref{030907a}$. According to Lemma \ref{mest}, ${\overline m}_1\approx \frac{\beta_0}{1-\beta_0}$ and tends to infinity as $\beta_0 \rightarrow 1-$, in fact ${\overline m}_1\cdot  \frac{\beta_0}{1-\beta_0}\to 1.$
This means that one can rewrite $\eqref{030913b}$ the following way
\begin{equation}\label{030914a}
 \partial^2_{\alpha}  \Theta^{\beta_0}(\beta_0,\beta_0)= \frac{2{\overline m}_1}{\beta_0^2}+\frac{{\overline m}_1({\overline m}_1-1)}{\beta_0^2} \cdot \left (\frac{\beta_0-1}{\beta_0} \right )+
C_{\alpha,\alpha}(\beta_0)
\end{equation}
with $|C_{\alpha,\alpha}(\beta_0)|<C_{\alpha,\alpha}^*$ where $C_{\alpha,\alpha}^*$ is not depending on $\beta_0$.

From $\eqref{030902a}$ we obtain
\begin{equation}\label{030915a}
\begin{split}
& \partial_{\alpha} \partial_{\beta}   \Theta^{\beta_0}(\beta_0,\beta_0)= \partial_{\beta} \partial_{\alpha}   \Theta^{\beta_0}(\beta_0,\beta_0)=
- \frac{{\overline m}_1+1}{\beta_0^2} - \frac{{\overline m}_1({\overline m}_1+1)}{\beta_0^2} \left(\frac{\beta_0-1}{\beta_0}\right) + \\ &\sum_{k=2}^{\infty}\left(\frac{-k({\overline m}_k+k)}{\beta_0^2}\left(\frac{\beta_0-1}{\beta_0}\right)^{k-1}-\frac{{\overline m}_k({\overline m}_k+k)}{\beta_0^2}\left(\frac{\beta_0-1}{\beta_0}\right)^k\right).
\end{split}
\end{equation}
We can estimate this as we treated $\eqref{030913b}$ to obtain $\eqref{030914a}$, namely
\begin{equation}\label{030915b}
\begin{split}
 \partial_{\alpha} \partial_{\beta}  \Theta^{\beta_0}(\beta_0,\beta_0)=-\frac{1+\omm_{1}}{\bbb_{0}^{2}}\left (1+\omm_{1}\left (\frac{\bbb_{0}-1}{\bbb_{0}} \right ) \right )+C_{\alpha,\beta}(\beta_0).
\end{split}
\end{equation}
with $|C_{\alpha,\beta}(\beta_0)|\leq C_{\alpha,\beta}^*$ where $C_{\alpha,\beta}^*$ is not depending on $\beta_0$. Finally, we use $\eqref{030902b}$ to obtain
\begin{equation}\label{030916a}
 \partial_{\beta}^2  \Theta^{\beta_0}(\beta_0,\beta_0)=\frac{(1+{\overline m}_1)(2+{\overline m}_1)}{\beta_0^2}\left(\frac{\beta_0-1}{\beta_0}\right)+\sum_{k=2}^{\infty}\frac{(k+{\overline m}_k)(k+{\overline m}_k+1)}{\beta_0^2}\left(\frac{\beta_0-1}{\beta_0}\right)^k.
\end{equation}
From $\eqref{030916a}$ we deduce
\begin{equation}\label{030916b}
 \partial_{\beta}^2  \Theta^{\beta_0}(\beta_0,\beta_0)=
 \frac{(1+{\overline m}_1)(2+{\overline m}_1)}{\beta_0^2}\left(\frac{\beta_0-1}{\beta_0}\right)
 +C_{\beta,\beta}(\beta_0)
\end{equation}
with $|C_{\beta,\beta}(\beta_0)|\leq C_{\beta,\beta}^*$ where $C_{\beta,\beta}^*$ is not depending on $\beta_0$. By using $\eqref{030914a}$, $\eqref{030915b}$ and $\eqref{030916b}$ we obtain the following for the second differential of $  \Theta^{\beta_0}$ at $\beta_0$.
\begin{equation}\label{030917a}
\begin{split}
d^2  \Theta^{\beta_0}&((\bbb_{0},\bbb_{0})(x,y))=\frac{1}{(1-\beta_0)\beta_0}\cdot \\
& \left(  \left( 2\omm_{1}\left (\frac{1-\bbb_{0}}{\bbb_{0}} \right )-\omm_{1}(\omm_{1}-1)\left (\frac{1-\bbb_{0}}{\bbb_{0}} \right )^{2}+ \left (\frac{1-\bbb_{0}}{\bbb_{0}} \right )C_{\alpha,\alpha}(\beta_0)\right )x^2 +\right.\\ &
\left. \left ((-1-\omm_{1})\left (\frac{1-\bbb_{0}}{\bbb_{0}} \right )\left(1-\omm_{1}\left (\frac{1-\bbb_{0}}{\bbb_{0}} \right )\right)+\left (\frac{1-\bbb_{0}}{\bbb_{0}} \right )C\saab(\bbb_{0}) \right ) 2xy+
\right.
\\ &
\left.
\left((1+\omm_{1})(2+\omm_{1})(-1)\left (\frac{1-\bbb_{0}}{\bbb_{0}} \right )^2
 +\left (\frac{1-\bbb_{0}}{\bbb_{0}} \right )C_{\beta,\beta}(\beta_0) \right)y^2\right).
\end{split}
\end{equation}
using $\omm_{1}\left (\frac{1-\bbb_{0}}{\bbb_{0}} \right )\to 1$ and $\left (\frac{1-\bbb_{0}}{\bbb_{0}} \right )\to 0$
as $\bbb_{0}\to 1-$
one can see that the local level set structure of $d^2  \Theta^{\beta_0}((\bbb_{0},\bbb_{0})(x,y))$ approximates the level
set structure of $x^2-y^2=0.$ By approximation this yields that the level set $  \Theta^{\beta_0}(\alpha,\beta)=0$ in a small neighborhood of
$(\beta_0,\beta_0)$ can be approximated by parts of the lines $y=x$ and $y=-(x-\beta_0)+\beta_0$. This implies the theorem. 
\end{proof}



\end{document}